\documentclass[11pt,english]{article}
\usepackage{amsthm}
\usepackage{amsmath}
\usepackage{amsmath,amscd}
\usepackage{amssymb}
\usepackage[T1]{fontenc}
\usepackage[adobe-utopia]{mathdesign}
\usepackage[a4paper]{geometry}
\usepackage{color}
\usepackage[pdftex]{graphicx}
\usepackage{xypic}
\usepackage{latexsym}
\usepackage{tikz}
\usetikzlibrary{matrix}
\usetikzlibrary{arrows,chains,matrix,positioning,scopes}
\usepackage{babel}
\usepackage{tikz}
\usetikzlibrary{arrows,chains,matrix,positioning,scopes}
\makeatletter
\tikzset{join/.code=\tikzset{after node path={%
\ifx\tikzchainprevious\pgfutil@empty\else(\tikzchainprevious)%
edge[every join]#1(\tikzchaincurrent)\fi}}}
\makeatother

\tikzset{>=stealth',/.append style={join},
         every join/.style={->}}

\makeatletter

\newtheorem{thm}{Theorem}[section]
  \newtheorem{prop}[thm]{Proposition}
  \newtheorem{example}[thm]{Example}
  \newtheorem{defn}[thm]{Definition}
   \newtheorem{cor}[thm]{Corollary}
 \newtheorem{lem}[thm]{Lemma}
  \newtheorem{rem}[thm]{Remark}
  
  \newtheorem*{rem*}{Remark}
  \newtheorem*{rems*}{Remarks}

\title{\textbf{Purity Relative to Classes of Finitely Presented Modules }}
\author{\textbf{Akeel Ramadan Mehdi }  \\School of Mathematics, University of Manchester, \\Manchester M13 9PL, England, United Kingdom.\\E-mail address: Akeel.Mehdi@postgrad.manchester.ac.uk}

\date{\textbf{8-5-2012}}

\AtBeginDocument{
  
}

\makeatother

\usepackage{babel}

\begin{document}
\newcommand{\To}{\longrightarrow}
\newcommand{\h}{\mathcal{H}}
\newcommand{\s}{\mathcal{S}}
\newcommand{\A}{\mathcal{A}}
\newcommand{\K}{\mathcal{K}}
\newcommand{\B}{\mathcal{B}}
\newcommand{\W}{\mathcal{W}}
\newcommand{\M}{\mathcal{M}}
\newcommand{\Lom}{\mathcal{L}}
\newcommand{\T}{\mathcal{T}}
\newcommand{\F}{\mathcal{F}}

\maketitle

\begin{abstract}

 We investigate purities determined by classes of finitely presented modules including the correspondence between purities for left and right modules.  We show some cases where purities determined by matrices of given sizes are different.  Then we consider purities over finite-dimensional algebras, giving a general description of the relative pure-injectives which we make completely explicit in the case of tame hereditary algebras.
\end{abstract}

\vspace{6pt}

\vspace{6pt}


\section*{Introduction:}$\quad \,\,\,\,\,\,$ Purity for modules over general rings was defined in \cite{Coh959} and many relative versions of purity have been considered since then.  We consider those purities which, like the original one, are determined by classes $S$ of finitely presented modules.  We present a number of characterisations of $S$-pure-exact sequences and of the associated classes of relatively projective and relatively injective modules.  We also show the relation between the purity for left modules which is determined by $S$ and the purity for right modules determined by $S$; this is said most directly in terms of the matrices presenting the modules in $S$.

Al-Kawarit  and Cauchot  \cite{AlkaCo11} gave conditions under which purities determined by matrices of certain sizes are different.  We obtain related results over semiperfect rings and we also consider this question in detail over finite-dimensional algebras.

Over finite-dimensional algebras we give a description of the $S$-pure-injective modules in terms of the type-definable category generated by $\tau S$ and, in the case of tame hereditary algebras, and using results from \cite{Pre09} and \cite{Rin98}, we give a complete description of these modules.

Finally we give a number of characterisations of rings whose indecomposable modules are $S$-pure-injective.

All rings in this paper are associative with unity, and all
modules are unital. We write $_{R}M$ $(M_{R})$ to indicate a left  (right) $R$-module,
and we use $R$-Mod denote the category of all left $R$-modules. The endomorphism ring of a module $M$ is denoted by {\rm End}$_{R}(M)$.  We use Add$(T)$ $($resp., add$(T))$ to denote the class of all modules that are direct
summands of direct sums $($resp. finite direct sums$)$ of modules
from $T.$ Also, we use Prod$(T)$ to denote the class of all modules that are direct summands of direct products of modules
from $T.$ We use the notation $\underset{n\times m}{M}(R)$ for the set
of all $n\times m$ matrices over $R$. All matrices in this paper are matrices with finitely many rows and finitely many
columns and all classes of modules are closed under isomorphisms.   A module is said to be finitely
presented if it is the factor module of a free  module of rank
$n$ modulo a $m$-generated submodule, for some $n,m\in\mathbb{Z}^{+}.$

 Let $S$ be a class of left $R$-modules. Following Warfield \cite{War69(b)}, an exact sequence $0\rightarrow A\overset{f}{\rightarrow}B\overset{g}{\rightarrow}C\rightarrow0$
of left $R$-modules is said to be $S$-pure if the sequence $0\rightarrow \textup{Hom}_{R}\left(M,A\right)\rightarrow \textup{Hom}_{R}\left(M,B\right)\rightarrow\textup{Hom}_{R}\left(M,C\right)\rightarrow0$
is exact, for all $M\in S;$ in this case $f$  is said to be an $S$-pure monomorphism and $g$ is said to be an $S$-pure epimorphism. Note that $S$-pure=$S\cup \{_RR\}$-pure. If  $S$=$R$-Mod  then a short exact sequence of modules is   $S$-pure if and only if it is   pure.   A module $M$ is said to be
$S$-pure-injective $($resp. $S$-pure-projective$)$, if $M$ is injective $($resp. projective$)$ relative to every $S$-pure exact
sequence of modules. Clearly the class  of $S$-pure-injective $($resp. $S$-pure-projective$)$ modules is closed under direct summands and direct products $($resp. direct sums$).$

This paper contains five sections.
In section 1,  many characterizations
and properties of $S$-purity, $S$-pure-injectivity and $S$-pure-projectivity  are given. For example, we prove that,  if $S$ is a class of finitely presented  modules then  a module
$M$ is $S$-pure-projective if and only if
it  is  projective relative to every $S$-pure
exact sequence $0\rightarrow K\rightarrow E\rightarrow F\rightarrow0$
   where  $E$ is $S$-pure-injective. Dually, $M$  is $S$-pure-injective if and only if  $M$ is  injective relative to every $S$-pure exact sequence $0\rightarrow K\rightarrow P\rightarrow L\rightarrow0$ where  $P$ is $S$-pure-projective.

In \cite{PuPrRo99} purity and  $S$-purity are compared. In particular, it is  proved that   $S$-purity and purity are equivalent if and only if   $S$-pure-injectivity  and pure-injectivity are equivalent if and only if  $R$-mod$\subseteq$add$(S\cup \{_RR\})$ \cite[Theorem 2.5, p.2136]{PuPrRo99}.
 In section 2 of this paper we  compare $S$-purity and $T$-purity for arbitrary  classes $S$ and $T$ of finitely presented left $R$-modules. For example, in Theorem~\ref{thm:2.1(2.15),(2.17),cor(2.18),cor(1.6),2report,p.14-16,p.5} we prove that if  $S$ and $T$ are classes of finitely presented left $R$-modules,  then the following statements are equivalent: $\left(1\right)$ every $T$-pure short exact sequence of
left $R$-modules is $S$-pure; $\left(2\right)$ $S\subseteq \textup{add}(T\cup\{_{R}R\});$  $\left(3\right)$ $D^*_{\mathcal G}\subseteq $Prod$(({D_{\mathcal H}}\cup{ R_{R}})^*)$ where $^*$ denotes the dual of a module $;\left(4\right)$ the corresponding assertions for  right modules. Also, in Proposition~\ref{thm:2.7(3.20),blue,p.5} we prove that if each  indecomposable direct summand of a module in $T$ has  local endomorphism ring and each module in $T$ is a
direct sum of indecomposable modules then   every $S$-pure short exact sequence modules
is $T$-pure if and only if each indecomposable direct summand of a module in $T$
is a direct summand of a module in $S\cup\{_{R}R\}.$

In section 3, we study $(n,m)$-purity over semiperfect rings.
In Theorem~\ref{thm:3.13(3.13),2report,p.20} we give a generalization of   \cite[Theorem 3.5(1), p.3888]{AlkaCo11} in which we prove that if  $(n,m)$ and $(r,s)$
are any two  pairs of positive integers such that  $n\neq r$ and if   one of the following two conditions is satisfied: $\left(a\right)$  $R$ is   semiperfect  and   there exists an  ideal
$I$ of $R$ with  $\textup{gen}(I_R)=$$\textup{max}\{n,r\}$  and $I\subseteq e_{j}R$
for some local idempotent $e_{j}$; $\left(b\right)$ $R$ is   Krull-Schmidt  and   there exists a right ideal
$I$ of $R$ with  $\textup{gen}(I)=$$\textup{max}\{n,r\}$ and $I\subseteq e_{j}R$
for some local idempotent $e_{j}$, then: $\left(1\right)$ $(m,n)$-purity and $(s,r)$-purity
of short exact sequences of left $R$-modules are not equivalent; $\left(2\right)$ $(n,m)$-purity and $(r,s)$-purity
of short exact sequences of right $R$-modules are not equivalent.

In section 4, we study purity over finite-dimensional algebras. Firstly, we compare purities over the Kronecker algebra over an algebraically closed field $k$. In Proposition~\ref{prop:4.4(new)} we prove that if  $R$ is  a finite-dimensional algebra over a field $k$ and it is  not of finite representation type, then for every $r\in \mathbb{Z}^{+}$, there is $n>r$ such that  $(\aleph_{0},n)$-purity$\neq (\aleph_{0},r)$-purity for  left $R$-modules. Let $\mathcal H$ be a set of
matrices over a tame hereditary
finite-dimensional algebra $R$ over a field $k$. Conditions under which the generic module is $L_{\mathcal H}$-pure-injective are given in Proposition~\ref{prop:4.10(prop 4.49, p.44)}. Finally,  we give a complete description of the full support topology closure of any  class of indecomposable finite-dimensional modules over a tame hereditary
finite-dimensional algebra $R$ over a field $k$.

In the last section  we give a condition on a left $R$-module $M$ such that every  $S$-pure submodule of $M$ is a direct summand and prove that  such module is a direct sum of indecomposable submodules. As a corollary of this result we give a characterizations of rings over which every indecomposable left $R$-module is $S$-pure-projective.

\section{Purities}

\subparagraph*{\textmd{Let $n,m\in \mathbb Z^{+}.$   An   $R$-module $M$ is said to be $(n,m)$-presented if it is the factor module of the module $R^n$ modulo an $m$-generated submodule. Let $H$ be an $n\times m$ matrix over  $R$. Then right $($resp. left$)$ multiplication by $H$ determines a
homomorphism $\rho_{H}:{_RR}^{n}\rightarrow{_RR}^{m}$ $($resp. $\lambda_{H}:R_{R}^{m}\rightarrow R_{R}^{n}).$  Then $H$ determines
the $(m,n)$-presented left $R$-module $R^{m}/$im$(\rho_H)$; we will denote
it by $L_{H}$. Also,  $H$ determines the $(n,m)$-presented
right $R$-module $R^{n}/$im$(\lambda_H);$ we will denote it by $D_{H}$.
Let ${ \mathcal H }$  be a set of matrices over a ring $R$; we will denote by
$L_{ \mathcal H }$ the class of left $R$-modules $\{L_{H}\mid H\in { \mathcal H } \}$
and by $D_{ \mathcal H }$ the class of right $R$-modules $\{D_{H}\mid H\in$${ \mathcal H }$$\}.$  In view  of proposition~\ref{cor:1.18(2.5),2report,p.11} below we may, where convenient,  interpret $L_{\emptyset}$ as $\{_RR\}$ and  $D_{\emptyset}$ as $\{R_R\}$, }}

\subparagraph*{\textmd{ The following theorem collects together and extends results from the literature $($in particular see  \cite{Fac98} and \cite{Wis91} $)$. A proof  can be found in the author's thesis \cite{Meh}.}}

\begin{thm}\label{thm:(1.3),2report,p3}
Let R be an algebra over a commutative ring K and let E be an injective
cogenerator for K-modules. Let $S$ be a class of finitely presented left $R$-modules, let ${ \mathcal H }$  be a set of matrices over $R$ such that $L_{ \mathcal H }$-purity=S-purity
and let $\Sigma:$$0\rightarrow A\overset{f}{\rightarrow}B\overset{g}{\rightarrow}C\rightarrow0$
be an exact sequence of left $R$-modules. Then the following statements
are equivalent.

$\left(1\right)$ $\Sigma$ is an $S$-pure exact sequence of left
R-modules.

$\left(2\right)$ For any two positive integers n,m,  for any $n\times m$
matrix H in $ \mathcal H$ and  for all $\bar{a}\in A^{n},$ if the matrix equation
H$\bar{x}$=f$\bar{a}$ has a solution in $B^{m}$ then the
equation H$\bar{x}$=$\bar{a}$ has a solution in $A^{m}.$

$\left(3\right)$ The sequence  $0\rightarrow M\otimes_{R}A\rightarrow M\otimes_{R}B\rightarrow M\otimes_{R}C\rightarrow0$
is exact, for all $M\in D_ \mathcal H.$

$\left(4\right)$ For any two positive integers n,m,  for any $n\times m$
matrix H in $ \mathcal H$ and  for all $\bar{c}\in C^{m},$ if $H\bar{c}$=$0,$
then there is $\bar{b}\in B^{m}$ with $g(\bar{b})$=$\bar{c}$ and
$H\bar{b}$=$0.$

$\left(5\right)$ For any two positive integers n,m and for any $n\times m$
matrix H in  $ \mathcal H,$ for every commutative diagram of left R-modules
\[\begin{tikzpicture}
  \matrix (m) [matrix of math nodes, row sep=3em, column sep=3em]
    { \ &R^n&R^m \\
      0 &A & B \\ };

\path[->]
(m-1-2)  edge node[auto]  {$\scriptstyle\ ^{\rho_{H}}$}  (m-1-3);
\path[->]

 (m-1-2) edge  node[left] {$\scriptstyle\alpha$} (m-2-2)
(m-2-1)  edge node[auto]  {$\scriptstyle$}  (m-2-2);
\path[->]
(m-2-2)  edge node[auto]  {$\scriptstyle\ f$}  (m-2-3);
\path[->]
(m-1-3) edge node[left]  {$\scriptstyle$} (m-2-3);

\end{tikzpicture}\]  there
exists a homomorphism $\beta:R^{m}\rightarrow A$ such that $\alpha$=$\beta \rho_H.$

$\left(6\right)$ The dual exact sequence of right R-modules $0\rightarrow C^{*}\rightarrow B^{*}\rightarrow A^{*}\rightarrow0$
is $D_{\mathcal H}$-pure, where $M^{*}$={\rm Hom}$_{K}(M,E).$\end{thm}

We retain the notation $M^{*}$   for the dual of a module with respect to    $_{K}E$ as above. Let $T$ be a class of left $R$-modules.
Note that if $S\subseteq T\subseteq R$-Mod then every $T$-pure   exact sequence of left $R$-modules
is $S$-pure, so  $S$-pure-injective  implies
$T$-pure-injective   and  $S$-pure-projective  implies
$T$-pure-projective.

\begin{prop}\label{cor:1.18(2.5),2report,p.11}(as \cite[Proposition 1, p.700]{War69(a)})
Let $S$ be a class of finitely presented left R-modules and let M be a left R-module.
Then:

$\left(1\right)$ There exists an $S$-pure exact sequence of left
R-modules $0\rightarrow K\rightarrow F\rightarrow M\rightarrow0$
with F being a direct sum of copies of modules in $S\cup\{_{R}R\}.$

$\left(2\right)$ \textup{Add}$(S\cup\{_{R}R\})$  is the class of  $S$-pure-projective left R-modules.

\end{prop}

\begin{cor}\label{cor:1.19(2.6),2report,p.11}
  Let $S$ be a class of finitely presented left R-modules  and let $ \mathcal H$    be a set of matrices over  R  such that $L_{ \mathcal H }$-purity=S-purity. Then for any left R-module N there is an $S$-pure
monomorphism $\alpha:N\rightarrow F^{*}$ such that F is a direct
sum of copies of modules in $D_{\mathcal H}\cup\{R_{R}\}.$ In particular,  see Theorem~\ref{thm:28(4.27),black,p.22}, $F^{*}$ is $S$-pure-injective.\end{cor}

\begin{proof}
Let $N$ be any left $R$-module. By the right hand version of Proposition~\ref{cor:1.18(2.5),2report,p.11}, there is a $D_{\mathcal H}$-pure
exact sequence of right $R$-modules \textit{$0\rightarrow G\overset{f}{\rightarrow}F\overset{g}{\rightarrow}N^{*}\rightarrow0$
}where \textit{ $F$ }is a direct sum of copies of modules in $D_{\mathcal H}\cup\{R_{R}\}.$
By the right hand version of Theorem~\ref{thm:(1.3),2report,p3}, the dual exact sequence
of left $R$-modules \textit{$0\rightarrow N^{**}\overset{g^{*}}{\rightarrow}F^{*}\overset{g^{*}}{\rightarrow}G^{*}\rightarrow0$
is $L_{\mathcal H}$-}pure. The canonical monomorphism $\varphi_{N}:N\rightarrow N^{**}$ is pure $($see, e.g.,  \cite[Corollary 1.30, p.17]{Fac98}$)$
and hence it is \textit{$L_{\mathcal H}$-}pure.  Since a composition of $L_{\mathcal H}$-pure monomorphisms clearly is $L_{\mathcal H}$-pure,
$g^{*}\varphi_{N}:N\rightarrow F^{*}$ is  an \textit{$L_{\mathcal H}$-}pure
monomorphism.\end{proof}

Let $S$ be a class of left $($or right$)$ $R$-modules. We use $S^*$ to denote the class $\{M^*\mid M\in S\}$.

\begin{thm}\label{thm:28(4.27),black,p.22}(as \cite[Theorem 1]{War69(b)})
Let S be a class of finitely presented left R-modules  and let ${ \mathcal H }$  be a set of matrices over R   such that $L_{ \mathcal H }$-purity=S-purity, then  \textup{Prod}$((D_{\mathcal H}$ $\cup\{R_{R}\})^{*})$ is the class of  $S$-pure-injective left R-modules.

 \end{thm}

\begin{proof}
Let $M$ be any $S$-pure-injective
left $R$-module. By  Corollary~\ref{cor:1.19(2.6),2report,p.11}, there
exists an $S$-pure, hence split,  monomorphism
$\alpha:M\rightarrow F^{*}$ where $F=\underset{i\in I}{\bigoplus}F_{i}$
 with $F_{i}\in D_{\mathcal H}\cup\{R_{R}\}.$  Since
$F^{*}=(\underset{i\in I}{\bigoplus}F_{i})^{*}\simeq \underset{i\in I}{\prod}F_{i}^{*}$
it follows that  $M\in  \textup{Prod}((D_{\mathcal H}\cup\{R_{R}\})^{*}).$

Conversely,   let $H\in\mathcal H$   and let $\Sigma:$ $0\rightarrow A\rightarrow B\rightarrow C\rightarrow0$
be any $L_H$-pure exact sequence of left $R$-modules. By Theorem~\ref{thm:(1.3),2report,p3}, the sequence
$D_{H}\otimes_{R}\Sigma:$$0\rightarrow D_{H}\otimes_{R}A\rightarrow D_{H}\otimes_{R}B\rightarrow D_{H}\otimes_{R}C\rightarrow0$
is  exact.  Since $E$ is an injective $K$-module,
the sequence $0\rightarrow {\rm Hom}_{K}(D_{H}\otimes_{R}C,E)\rightarrow {\rm Hom}_{K}(D_{H}\otimes_{R}B,E)\rightarrow {\rm Hom}_{K}(D_{H}\otimes_{R}A,E)\rightarrow0$
is exact.  This is isomorphic to the sequence $0\rightarrow {\rm Hom}_{R}(C,{\rm Hom}_{K}(D_{H},E))\rightarrow {\rm Hom}_{R}(B,{\rm Hom}_{K}(D_{H},E))\rightarrow {\rm Hom}_{R}(A,{\rm Hom}_{K}(D_{H},E))\rightarrow0$. That is, the sequence $0\rightarrow {\rm Hom}_{R}(C,D_{H}^{*})\rightarrow {\rm Hom}_{R}(B,D_{H}^{*})\rightarrow {\rm Hom}_{R}(A,D_{H}^{*})\rightarrow0$
is exact. Therefore, $D^{*}_{H}$ is    $L_H$-pure-injective.  By, for instance,   \cite[Theorem 3.2.9, p.77]{EnJe00}, $R_R^*$ is injective  and  thus  each module in $(D_{\mathcal H}\cup\{R_{R}\})^{*}$
is $S$-pure-injective.  It follows that every module in \textup{Prod}$((D_{\mathcal H}$ $\cup\{R_{R}\})^{*})$ is $S$-pure-injective.
\end{proof}

\begin{prop}\label{prop:26(2.13),2report,p.13}   Let S be a class of finitely presented left R-modules, let ${ \mathcal H }$  be a set of matrices over R   such that $L_{ \mathcal H }$-purity=S-purity and let $\Sigma:$ $0\rightarrow A\rightarrow B\rightarrow C\rightarrow0$
be any exact sequence of left R-modules. Then the following statements
are equivalent:

$\left(1\right)$ $\Sigma$ is $S$-pure.

$\left(2\right)$ Every $S$-pure-injective left R-module is injective
relative to $\Sigma.$

$\left(3\right)$ $D^{*}_{H}$ is injective relative to $\Sigma,$
for all $H\in{ \mathcal H }.$

$\left(4\right)$ Every $S$-pure-projective left R-module is projective
relative to $\Sigma.$ \end{prop}

\begin{proof}
$\left(1\right)\Rightarrow\left(2\right)$ and $\left(1\right)\Rightarrow\left(4\right)$
are obvious  and $\left(2\right)\Rightarrow\left(3\right)$  is immediate  from  Theorem~\ref{thm:28(4.27),black,p.22}.

$\left(3\right)\Rightarrow\left(1\right)$ Let $H\in\mathcal H.$ By hypothesis,
the sequence\[
0\rightarrow \textup{Hom}_{R}(C,\textup{Hom}_{K}(D_{H},E))\rightarrow \textup{Hom}_{R}(B,\textup{Hom}_{K}(D_{H},E))\rightarrow \textup{Hom}_{R}(A,\textup{Hom}_{K}(D_{H},E))\rightarrow0,\] equivalently,   the sequence\[
0\rightarrow \textup{Hom}_{K}(D_{H}\otimes_{R}C,E)\rightarrow\textup{Hom}_{K}(D_{H}\otimes_{R}B,E)\rightarrow \textup{Hom}_{K}(D_{H}\otimes_{R}A,E)\rightarrow0\]
 is exact. Since $E$ is an injective cogenerator for $K$-modules, it follows  $($see \cite[Lemma 3.2.8, p.77]{EnJe00}$)$ that
the sequence $0\rightarrow D_{H}\otimes_{R}A\rightarrow D_{H}\otimes_{R}B\rightarrow D_{H}\otimes_{R}C\rightarrow0$
is exact. Thus $\Sigma$ is  $S$-pure.

$\left(4\right)\Rightarrow\left(1\right)$ This is immediate from Proposition~\ref{cor:1.18(2.5),2report,p.11},
and the definition of  $S$-pure exact sequence.\end{proof}

\begin{thm}\label{thm:20(2.7),2report,p.11}
Let $S$ be  a class of finitely presented left $R$-modules. Then  for a left R-module M:

$\left(1\right)$  M  is S-pure-projective if and only if
it  is  projective relative to every $S$-pure
exact sequence $0\rightarrow K\rightarrow E\rightarrow F\rightarrow0$
of left R-modules where  E is $S$-pure-injective;

$\left(2\right)$  M  is S-pure-injective if and only if  M is  injective relative to every $S$-pure exact
sequence $0\rightarrow K\rightarrow P\rightarrow L\rightarrow0$ of
left R-modules  where  P is $S$-pure-projective.\end{thm}

\begin{proof}
$\left(1\right)$ $(\Rightarrow)$ is obvious.

$\left(\Leftarrow\right)$  Let $0\rightarrow A\overset{\mu}{\rightarrow}B\overset{\nu}{\rightarrow}C\rightarrow0$
be any $S$-pure exact sequence of left $R$-modules. By  Corollary~\ref{cor:1.19(2.6),2report,p.11} and Theorem~\ref{thm:28(4.27),black,p.22}, there is an $S$-pure exact sequence $0\rightarrow B\overset{\lambda}{\rightarrow}P\overset{\rho}{\rightarrow}N\rightarrow0$
of left $R$-modules where  $P$ is  $S$-pure-injective. We
have the following pushout diagram:

\[\begin{tikzpicture}
  \matrix (m) [matrix of math nodes, row sep=2em, column sep=2em]
    {  &   & 0 & 0& \\  0& A  & B & C&0     \\ 0& A  & P & D&0        \\  &   & N & N   \\&  & 0 & 0&   \\ };

 { [start chain] \;

    \chainin (m-2-2);
    { [start branch=B] \chainin (m-3-2)
        [join={node[right] {$\scriptstyle I_A$}}];}

  }
  { [start chain] \;
    \chainin (m-1-4);
    { [start branch=B] \chainin (m-2-4)
        [join={node[right] {$$}}];}
    ;
    \chainin (m-2-4);
    { [start branch=B] \chainin (m-3-4)
        [join={node[right] {$\scriptstyle \varphi$}}];}
 ;
    \chainin (m-3-4);
    { [start branch=B] \chainin (m-4-4)
        [join={node[right] {$\scriptstyle \delta$}}];}
 ;
    \chainin (m-4-4);
    { [start branch=B] \chainin (m-5-4)
        [join={node[right] {$$}}];}
  }

 { [start chain] \;
    \chainin (m-1-3);
    { [start branch=B] \chainin (m-2-3)
        [join={node[right] {$$}}];}
    ;
    \chainin (m-2-3);
    { [start branch=B] \chainin (m-3-3)
        [join={node[right] {$\scriptstyle \lambda$}}];}
 ;
    \chainin (m-3-3);
    { [start branch=B] \chainin (m-4-3)
        [join={node[right] {$\scriptstyle \rho$}}];}
 ;
    \chainin (m-4-3);
    { [start branch=B] \chainin (m-5-3)
        [join={node[right] {$$}}];}
  }

  { [start chain] \chainin (m-2-1);
    \chainin (m-2-2) [join={node[above] {}}];
    \chainin (m-2-3) [join={node[above] {$\scriptstyle\mu$}}];
    \chainin (m-2-4) [join={node[above] {$\scriptstyle\nu$ }}];
 \chainin (m-2-5) [join={node[above] {}}];
     }
{ [start chain] \chainin (m-3-1);
    \chainin (m-3-2) [join={node[above] {}}];
    \chainin (m-3-3) [join={node[above] {$\scriptstyle\alpha$}}];
    \chainin (m-3-4) [join={node[above] {$\scriptstyle\beta$ }}];
 \chainin (m-3-5) [join={node[above] {}}];
     }
{ [start chain];
    \chainin (m-4-3) [join={node[above] {$\alpha$}}];
    \chainin (m-4-4) [join={node[above] {$\scriptstyle I_N $ }}];
];
     }
\end{tikzpicture}\]

Since  $\mu$ and $\lambda$ are $S$-pure
$R$-monomorphisms so is   $\lambda\mu$. Since $\alpha$=$\lambda\mu,$
the exact sequence  $0\rightarrow A\overset{\alpha}{\rightarrow}P\overset{\beta}{\rightarrow}D\rightarrow0$
is  $S$-pure.  Let $\psi\in {\rm Hom}_{R}(M,C).$
 By hypothesis, there is $\gamma\in {\rm Hom}_{R}(M,P)$
such that $\beta\gamma$=$\varphi\psi.$ We have  $\rho\gamma$=$\delta\beta\gamma$=$\delta\varphi\psi =0$ so im$(\gamma)\subseteq$ ker$(\rho)$=im$(\lambda)$ and hence
$\gamma$=$\lambda\gamma^{^{\prime}}$ for some $\gamma^{^{\prime}}\in {\rm Hom}_{R}(M,B).$
Then we have $\varphi\nu\gamma^{^{\prime}}$=$\beta\lambda\gamma^{^{\prime}}$=$\beta\gamma$=$\varphi\psi.$
 Since $\varphi$ is a monomorphism, $\nu\gamma^{^{\prime}}$=$\psi.$ Hence $M$ is  $S$-pure-projective.

$\left(2\right)$ The proof is  dual to that of  $\left(1\right).$
\end{proof}

\begin{cor}\label{cor:22(3.25),blue,p.10}
Let $S$ be  a class of finitely presented left $R$-modules. Then the following statements are equivalent:

$\left(1\right)$ For every $S$-pure exact sequence $0\rightarrow N\rightarrow M\rightarrow K\rightarrow0$
of left R-modules, if M is  $S$-pure-projective, then N is  $S$-pure-projective.

$\left(2\right)$ For every $S$-pure exact sequence $0\rightarrow N\rightarrow M\rightarrow K\rightarrow0$
of left R-modules, if M is  $S$-pure-injective, then K is  $S$-pure-injective.\end{cor}

\begin{proof}
$\left(1\right)\Rightarrow\left(2\right)$ Let \textit{$0\rightarrow N\overset{\nu}{\rightarrow}M\overset{\mu}{\rightarrow}K\rightarrow0$
}be any $S$-pure exact sequence of left $R$-modules where $M$
is  $S$-pure-injective. Let \textit{$0\rightarrow A\overset{\alpha}{\rightarrow}B\overset{\beta}{\rightarrow}C\rightarrow0$
}be any\textit{ }$S$-pure exact sequence of left $R$-modules where
$B$ is  $S$-pure-projective. By hypothesis, $A$ is  $S$-pure-projective.
Let $f:A\rightarrow K$ be any $R$-homomorphism. Thus there is an
$R$-homomorphism $g:A\rightarrow M$ such that $\mu g=f.$ Since
$M$ is  $S$-pure-injective, there is an $R$-homomorphism $h:B\rightarrow M$
such that $h\alpha=g.$ Put $\lambda=\mu h,$ thus $\lambda\alpha=(\mu h)\alpha=\mu(h\alpha)=\mu g=f.$
Hence $K$ is injective relative to every $S$-pure exact sequence
\textit{$0\rightarrow A\rightarrow B\rightarrow C\rightarrow0$ } where $B$ is  $S$-pure-projective. By Theorem~\ref{thm:20(2.7),2report,p.11}, $K$\textit{ }is  $S$-pure-injective.

$\left(2\right)\Rightarrow\left(1\right)$ Let \textit{$0\rightarrow N\overset{\nu}{\rightarrow}M\overset{\mu}{\rightarrow}K\rightarrow0$
}be any $S$-pure exact sequence of left $R$-modules where $M$
is  $S$-pure-projective. Let \textit{$0\rightarrow A\overset{\alpha}{\rightarrow}B\overset{\beta}{\rightarrow}C\rightarrow0$
}be any\textit{ }$S$-pure exact sequence of left $R$-modules where
$B$ is  $S$-pure-injective. By hypothesis, $C$ is  $S$-pure-injective.
Let $f:N\rightarrow C$ be any $R$-homomorphism. Thus there is an
$R$-homomorphism $g:M\rightarrow C$ such that $g\nu=f.$ Since $M$
is  $S$-pure-projective, there is an $R$-homomorphism $h:M\rightarrow B$
such that $\beta h=g.$ Put $\lambda=h\nu,$ thus $\beta\lambda=\beta h\nu=g\nu=f.$
Hence $N$ is projective relative to every $S$-pure exact sequence
\textit{$0\rightarrow A\rightarrow B\rightarrow C\rightarrow0$ }of
left $R$-modules where $B$ is  $S$-pure-injective. By  Theorem~\ref{thm:20(2.7),2report,p.11}, $N$\textit{ }is  $S$-pure-projective.
\end{proof}

\section{Comparing purities}

\begin{thm}\label{thm:2.1(2.15),(2.17),cor(2.18),cor(1.6),2report,p.14-16,p.5} Let S and T be classes of finitely presented left R-modules and let  $\mathcal{G}$ and $\mathcal{H}$ be sets of
matrices over  R such that $L_{ \mathcal G}$-purity=S-purity and $L_{ \mathcal H }$-purity=T-purity. Then the following statements are equivalent.

$\left(1\right)$ Every T-pure short exact sequence of
left R-modules is S-pure.

$\left(2\right)$ Every T-pure exact sequence $0\rightarrow A\rightarrow B\rightarrow C\rightarrow0$
of left R-modules where  B is  T-pure-injective is S-pure.

$\left(3\right)$ Every S-pure-projective left R-module is
T-pure-projective.

$\left(4\right)$ $S\subseteq \textup{add}(T\cup\{_{R}R\}).$

$\left(5\right)$ Every T-pure exact sequence $0\rightarrow A\rightarrow B\rightarrow C\rightarrow0$
of left R-modules where B is  T-pure-projective is S-pure.

$\left(6\right)$ Every S-pure-injective left R-module is
T-pure-injective.

$\left(7\right)$ $D^*_{\mathcal G}\subseteq $Prod$(({D_{\mathcal H}}\cup{ R_{R}})^*).$

$\left(8\right)$ The corresponding assertions for  right modules.\end{thm}

\begin{proof}
$\left(1\right)\Rightarrow\left(2\right)$ and $\left(1\right)\Rightarrow\left(5\right)$    are obvious.

$\left(2\right)\Rightarrow\left(3\right)$ Let $M$ be any $ S$-pure-projective
left $R$-module and let $\Sigma:$ $0\rightarrow A\rightarrow B\rightarrow C\rightarrow0$
be any $T$-pure exact sequence of left $R$-modules where
$B$ is  $T$-pure-injective. By hypothesis, $\Sigma$
is $S$-pure and hence the sequence $0\rightarrow  \textup{Hom}_{R}\left(M,A\right)\rightarrow  \textup{Hom}_{R}\left(M,B\right)\rightarrow  \textup{Hom}_{R}\left(M,C\right)\rightarrow0$
is exact. Thus $M$ is  projective relative to every $T$-pure
exact sequence $\Sigma:$ $0\rightarrow A\rightarrow B\rightarrow C\rightarrow0$
of left $R$-modules where  $B$ is  $T$-pure-injective.
By Theorem~\ref{thm:20(2.7),2report,p.11}, $M$ is $T$-pure-projective.

$\left(3\right)\Rightarrow\left(4\right)$ This follows by    Proposition~\ref{cor:1.18(2.5),2report,p.11}.

$\left(4\right)\Rightarrow\left(1\right)$ Let $\Sigma:$ $0\rightarrow A\rightarrow B\rightarrow C\rightarrow0$
be any $T$-pure exact sequence of left $R$-modules and let
$M\in S$. By assumption and Proposition~\ref{cor:1.18(2.5),2report,p.11} $M$
is $T$-pure-projective. Thus the sequence
$0\rightarrow \textup{Hom}_{R}\left(M,A\right)\rightarrow \textup{Hom}_{R}\left(M,B\right)\rightarrow  \textup{Hom}_{R}\left(M,C\right)\rightarrow0$
is exact.   Therefore  $\Sigma$ is $S$-pure.

$\left(5\right)\Rightarrow\left(6\right)$ Let $M$ be any $S$-pure-injective
left $R$-module and let $\Sigma:$ $0\rightarrow A\rightarrow B\rightarrow C\rightarrow0$
be any $T$-pure exact sequence of left $R$-modules where
$B$ is  $T$-pure-projective. By hypothesis, $\Sigma$
is $S$-pure and hence the sequence $0\rightarrow \textup{Hom}_{R}\left(C,M\right)\rightarrow  \textup{Hom}_{R}\left(B,M\right)\rightarrow  \textup{Hom}_{R}\left(A,M\right)\rightarrow0$
is exact. It follows by  Theorem~\ref {thm:20(2.7),2report,p.11} that $M$ is $T$-pure-injective.

$\left(6\right)\Rightarrow\left(7\right)$ Let
 $M\in D^*_{\mathcal G},$ thus $M$ is an $S$-pure-injective
left $R$-module $($by Theorem~\ref{thm:28(4.27),black,p.22}). By hypothesis, $M$ is $T$-pure-injective so
by Theorem~\ref {thm:28(4.27),black,p.22} we have that  $M\in$Prod$(({D_{\mathcal H}}\cup{ R_{R}})^*)$.

$\left(7\right)\Rightarrow\left(1\right)$  Let $\Sigma:$ $0\rightarrow A\rightarrow B\rightarrow C\rightarrow0$
be any $T$-pure exact sequence of left $R$-modules. Let
$G\in\mathcal G,$ thus by hypothesis, $D^*_{G} \in$ Prod$(({D_{\mathcal H}}\cup{ R_{R}})^*),$   hence  $D^*_{G}$ is $T$-pure-injective, in particular
  $D^*_{G}$ is injective
relative to $\Sigma$.  By Proposition~\ref {prop:26(2.13),2report,p.13}, $\Sigma$ is $S$-pure.

$\left(1\right)\Rightarrow\left(8\right)$ Let $\Sigma:$$0\rightarrow A\rightarrow B\rightarrow C\rightarrow0$
be any  $D_\mathcal{H}$-pure exact sequence of right $R$-modules. By the right
hand version of  Theorem~\ref {thm:(1.3),2report,p3}, the exact sequence of
left $R$-modules $\Sigma^{*}:$ $0\rightarrow C^{*}\rightarrow B^{*}\rightarrow A^{*}\rightarrow0$
is $T$-pure.  By hypothesis, $\Sigma^{*}$
is  $S$-pure and hence by   Theorem~\ref {thm:(1.3),2report,p3} again,  $\Sigma$ is $D_{\mathcal G}$-pure.

$\left(8\right)\Rightarrow\left(1\right)$ This follows by right/left symmetry.

 \end{proof}

The following corollary is immediately obtained from Theorem~\ref{thm:2.1(2.15),(2.17),cor(2.18),cor(1.6),2report,p.14-16,p.5}.

\begin{cor}\label{cor:2.2(2.15(a)) general case of (cor.(2.19),2report,p16)}
Let S and T be classes of finitely presented left R-modules and let  $\mathcal{G}$ and $\mathcal{H}$ be sets of
matrices over  R such that $L_{ \mathcal G}$-purity=S-purity and $L_{ \mathcal H }$-purity=T-purity. Then the following statements are equivalent:

$\left(1\right)$ T-purity =S-purity for short
exact sequences of left R-modules.

$\left(2\right)$ S-pure-projectivity=T-pure-projectivity
for left R-modules.

$\left(3\right)$ \textup{add}$(S\cup\{_{R}R\})$=
\textup{add}$(T\cup\{_{R}R\}).$

$\left(4\right)$ S-pure-injectivity=T-pure-injectivity
for left R-modules.

$\left(5\right)$ \textup{Prod}$(\{D^*_{G}\mid G\in\mathcal G\cup\{\underset{1\times1}{0}\}\})$=\textup{Prod}$(\{ D^*_{H}\mid H\in\mathcal{H\cup}\{\underset{1\times1}{0}\}\}).$

$\left(6\right)$ The corresponding assertions on the right.\end{cor}

 A short exact sequence $\left(\Sigma\right)$ of left $($resp. right$)$ $R$-modules is called $(m,n)$-pure if it remains exact when tensored with any $(m,n)$-presented right $($resp. left$)$ $R$-module. A left $R$-module $M$ is said to be $(m,n)$-pure-projective $($resp.  $(m,n)$-pure-injective$)$ if it is projective $($resp. injective$)$  relative to every $(m,n)$-pure exact sequence of left $R$-modules . A short exact sequence $\left(\Sigma\right)$ of left $($or right$)$  $R$-modules is called $(\aleph_{0},n)$-pure exact $($resp. $(m,\aleph_{0})$-pure exact$)$ if, for each positive integer $m$ $($resp. $n)$ $\left(\Sigma\right)$ is $(m,n)$-pure \cite{AlkaCo11}.
Observe that the $(m,n)$-pure exact
sequences of left $R$-modules are
exactly the $L_{\mathcal H}$-pure  exact sequences, where $\mathcal H$=$M_{m\times n}(R)$, and the $(n,m)$-pure exact
sequences of right $R$-modules are exactly the $D_{\mathcal H}$-pure  exact sequences of right modules. Also, $(\aleph_{0},n)$-pure
exact sequences of left $R$-modules
are exactly the  $L_{\mathcal H}$-pure  exact sequences, where $\mathcal H$= $\underset{m\in\mathbb{Z^+}}{\bigcup} M_{m\times n}(R)$ and then the  $(n,\aleph_{0})$-pure
exact sequences of right $R$-modules are exactly the $D_{\mathcal H}$-pure exact
sequences. Note that for left modules $(n,m)$-presented implies $(m,n)$-pure-projective, where as for right modules $(n,m)$-presented implies $(n,m)$-pure-projective. For all $n,m,s,t\in \mathbb Z^{+}$ with $n\geq s$ and $m\geq t$, since every $(t,s)$-presented right  $R$-module is $(m,n)$-presented it follows that every $(m,n)$-pure exact sequence of left $R$-modules is $(t,s)$-pure.

\begin{cor}\label{cor:2.3(2.15(b)),p.3}
Let $n,m,s,t\in \mathbb Z^{+}$.  Then the following statements are equivalent:

$\left(1\right)$ Every $(m,n)$-pure short exact sequence of
left $R$-modules is $(s,t)$-pure.

$\left(2\right)$ Every $(n,m)$-pure short exact sequence of
right $R$-modules is $(t,s)$-pure.

$\left(3\right)$ Every  $(s,t)$-pure-projective  $($resp. $(s,t)$-pure-injective$)$ left $R$-module  is $(m,n)$-pure-pr- ojective $($resp. $(m,n)$-pure-injective$)$.

$\left(4\right)$  Every  $(t,s)$-presented left $R$-module is   in \textup{add}$(\{M \mid M$ is an  $(n,m)$-presented left $R$-module$\}).$

$\left(5\right)$ Every  $(s,t)$-presented right $R$-module is   in \textup{add}$(\{M \mid M$ is an  $(m,n)$-presented right $R$-module$\}).$\end{cor}

\begin{proof}
 Take  $S=L_{\mathcal{G}}$ and  $T=L_{\mathcal{H}}$   where  $\mathcal G$=${M_{s\times t}}(R)$
 and $\mathcal H$=${M_{m\times n}}(R)$
and apply Theorem~\ref{thm:2.1(2.15),(2.17),cor(2.18),cor(1.6),2report,p.14-16,p.5}.\end{proof}

\begin{prop}\label{thm:2.7(3.20),blue,p.5}
Let S and T be classes of finitely presented left R-modules. Consider the following  statements:

$\left(1\right)$ Every S-pure short exact sequence of left $R$-modules
is T-pure.

$\left(2\right)$ Each indecomposable direct summand of a module in T is in \textup{add}$(S\cup\{_{R}R\}).$

$\left(3\right)$ Each indecomposable direct summand of a module in T
is a direct summand of a module in $S\cup\{_{R}R\}.$

Then $\left(1\right)$ implies $\left(2\right)$ and

$\left(a\right)$ If each  indecomposable direct summand of a module in T
has  local endomorphism ring then $\left(2\right)$ implies $\left(3\right).$

$\left(b\right)$ If each module in T is a
direct sum of indecomposable modules then $\left(3\right)$
implies $\left(1\right).$ \end{prop}

\begin{proof}
$\left(1\right)\Rightarrow\left(2\right)$ This follows by  Theorem~\ref{thm:2.1(2.15),(2.17),cor(2.18),cor(1.6),2report,p.14-16,p.5}.

$\left(a\right)$ Assume that each  indecomposable direct summand $M$ of a module in $T$
has  local endomorphism ring, thus by hypothesis, $M\in$add$(S\cup\{_{R}R\}).$
 Suppose that  $M$  is
a direct summand of $\underset{i\in I}{\bigoplus}F_{i}$ where $F_{i}\in S\cup\{_{R}R\},$
for all $i\in I$ and $I$ is a finite set and let $B$ be a submodule
of $\underset{i\in I}{\bigoplus}F_{i}$ such that $M\oplus B=\underset{i\in I}{\bigoplus}F_{i}.$
Since  End$_{R}(M)$ is local we have $($see, e.g., \cite[Theorem 2.8, p.37]{Fac98}$)$
 that $M$ has the finite exchange property. So  $($see, e.g., \cite[Lemma 2.7, p.37]{Fac98}$)$
there is an index $j\in I$ and a direct sum decomposition $F_{j}=B_{j}\oplus C_{j}$
of $F_{j}$  with
$M\simeq C_{j}$.
Hence $M$ is a direct summand of a module in $S\cup\{_{R}R\}.$

$\left(b\right)$ This  follows directly using Proposition~\ref{prop:26(2.13),2report,p.13}.
\end{proof}

A ring $R$ is said to be Krull-Schmidt if every finitely
presented left (or right) $R$-module is a direct sum of modules with
local endomorphism rings $($see \cite[p.97]{Fac98}$)$.

\begin{cor}\label{cor:2.8(3.21),blue,p.6}
 Let R be a left Krull-Schmidt ring and
let n,m be positive integers. Then the following statements are equivalent:

$\left(1\right)$ $(m,n)$-purity=$(\aleph_{0},n)$-purity for short exact
sequences of left $R$-modules.

$\left(2\right)$ For each $s\in \mathbb Z^{+},$ each indecomposable $(n,s)$-presented
left R-module is a direct summand of an $(n,m)$-presented
left R-module.

$\left(3\right)$ $(n,m)$-purity=$(n,\aleph_{0})$-purity for short exact
sequences of right $R$-modules.

$\left(4\right)$ For each $s\in \mathbb Z^{+},$ each indecomposable $(s,n)$-presented
right R-module is a direct summand of an $(m,n)$-presented
right R-module.
\end{cor}

\begin{proof} Put $S=L_\mathcal G$ and $T=L_\mathcal H$,  where
$\mathcal G$=${M_{m\times n}}(R)$ and $\mathcal{H}=\underset{t\in\mathbb{Z^+}}{\bigcup} M_{t\times n}(R).$
Since $R$ is Krull-Schmidt,  each  indecomposable direct summand of a module in $T$
has  local endomorphism ring and each module in $T$ is a
direct sum of indecomposable modules.
Hence the result follows on  applying  Proposition~\ref{thm:2.7(3.20),blue,p.5} and Corollary~\ref{cor:2.3(2.15(b)),p.3}.

\end{proof}

\section{$(m,n)$-Purity over semiperfect rings}

\subparagraph*{\textmd{  Let $M$ be a finitely presented left (or right) $R$-module, we denote
by gen$(M)$ its minimal number of generators and by rel$(M)$ the
minimal number of relations on these generators. Therefore there is an exact sequence  $ R^{\textup{rel}(M)}\rightarrow R^{\textup{gen}(M)}\rightarrow M\rightarrow0$}} and it follows easily that rel$(M)$  is the minimal number of relations on any generating set of $M$.

\begin{rem}\label{rem:3.1(3.1),2report,p.16}
Let M be a finitely presented left R-module and let N be a direct summand of M. Then it is easy to see that {\rm gen}$(N)\leq${\rm gen}$(M)$ and {\rm rel}$(N)\leq${\rm rel}$(M)+${\rm gen}$(M).$
\end{rem}

\begin{prop}\label{prop:3.2(3.2),2report,p.16}
Let H be any matrix over a ring R such that  {\rm End}$_{R}(L_{H})$ is  local and $L_H$ is not projective. Set  $\mathcal H$=$\bigcup\{{M_{r\times q}}(R) \mid$ q<{\rm gen}$(L_{H})$ or $r+q<${\rm rel}$(L_{H})$ $\}.$  Then $L_{H}$ is an  $L_H$-pure-projective left R-module
which is not $L_{\mathcal H}$-pure-projective and hence not $L_{\mathcal G}$-pure-projective for any ${\mathcal G}\subseteq{\mathcal H}$. In particular $L_H$-purity and $L_{\mathcal H}$-purity are not equivalent.\end{prop}

\begin{proof}  By  Proposition~\ref{cor:1.18(2.5),2report,p.11}, $L_{H}$ is   $L_H$-pure-projective and, if $L_H$ is    $L_{\mathcal H}$-pure-projective, then  $L_{H}\in$ Add$(L_{\mathcal H}\cup\{_{R}R\}).$
Since End$_{R}(L_H)$ is  local, $L_H$
is, as in Proposition~\ref{thm:2.7(3.20),blue,p.5},  a direct summand of a module in $L_{\mathcal H}\cup\{_{R}R\}.$  Thus  either $L_H$ is a direct summand of $L_G,$ where $\underset{r\times q}{G}\in\mathcal H$ or $L_H$ is projective. If $L_H$ is a direct summand of $L_G,$  by Remark~\ref{rem:3.1(3.1),2report,p.16}, {\rm gen}$(L_H)\leq$gen$(L_G)\leq q$
and rel$(L_H)\leq$rel$(L_G)+$
gen$(L_G)\leq r+q$  and this contradicts  $G\in\mathcal H$.
  \end{proof}

Note that if  $M$ is a left $R$-module, $I$ is  a left ideal of $R$ and  $\alpha\in {\rm End}_{R}(M)$
then there is an induced homomorphism $\overline{\alpha}:M/IM\rightarrow M/IM$ which is an isomorphism if  $\alpha$ is an isomorphism.

Let $R$ be a ring and let $J$ be its Jacobson radical. Recall
that $R$ is semiperfect if $R/J$ is semisimple and idempotents lift
modulo $J$. Say that an idempotent $e\in R$ is local if $eRe$
is a local ring. We have $($e.g., \cite[42.6, p.375]{Wis91}$)$ that $R$
semiperfect if and only if $R$=$e_{1}R\oplus e_{2}R\oplus\cdots\oplus e_{n}R,$
for local orthogonal idempotents $e_{i}$.

\begin{lem}\label{lem:3.10(3.10),2report,p.18} Let $m\in \mathbb Z^{+}$.
Suppose that  one of the following two conditions is satisfied.

$\left(1\right)$ The  ring $R$ is   semiperfect  and   $I$ is  a nonzero ideal
 with $\textup{gen}(I_R)=m$ and  $I\subseteq e_{j}R$
for some local idempotent $e_j$ of R.

$\left(2\right)$ The ring $R$ is   Krull-Schmidt  and  $I$ is a nonzero right ideal
with $\textup{gen}(I)=m$ and  $I\subseteq e_{j}R$ for some local idempotent $e_j$ of R.

 Then $e_{j}R/I$ is a finitely presented right R-module
with {\rm gen}$(e_{j}R/I)$=1, {\rm rel}$(e_{j}R/I)$=m and ${\rm End}_{R}(e_{j}R/I)$
is a local ring.\end{lem}

\begin{proof}
Let $P$=$e_{j}R.$ Then {\rm gen}$(P/I)$=$1$ and clearly {\rm rel}$(P/I)$={\rm gen}$(I)$=$m.$

In case $\left(1\right)$: Since ${\rm End}_{R}(e_{j}R)\simeq e_{j}Re_{j}$
it follows that  ${\rm End}_{R}(P)$ is a local ring.  Let $\alpha\in {\rm End}_{R}(P/I)$ and consider the following diagram:

\[\begin{tikzpicture}
  \matrix (m) [matrix of math nodes, row sep=3em, column sep=3em]
    { \ P&P/I \\
      P &P/I \\ };

\path[->]
(m-1-1)  edge node[auto]  {$\scriptstyle\pi$}  (m-1-2);
\path[->]

 (m-1-2) edge  node[auto] {$\scriptstyle\alpha$} (m-2-2)
(m-2-1)  edge node[auto]  {$\scriptstyle\pi$}  (m-2-2);
\path[dashed ][->]
(m-1-1) edge node[left]  {$\scriptstyle\alpha^{'}$} (m-2-1);
\end{tikzpicture}\] where $\pi$ is the natural epimorphism. By projectivity of $P,$ there exists an $R$-homomorphism  $\alpha^{'}:P\rightarrow P$  such that $\pi \alpha^{'}$=$\alpha \pi$ and
$\alpha^{'}(I) \subseteq I.$ Since ${\rm End}_{R}(P)$ is a local ring,  either $\alpha^{'}$ or $1_{P}$-$\alpha^{'}$ is an isomorphism. The inverse of that isomorphism will, as noted above, induce an isomorphism on $P/I=P/PI$  which will be an inverse of $\alpha$ or $1_{(P/I)}-\alpha$, as appropriate.  Hence ${\rm End}_{R}(e_{j}R/I)$ is a local ring.

In case $\left(2\right)$: Since $e_{j}R$ is a local right $R$-module,
every homomorphic image of $e_{j}R$ is indecomposable \cite[Proposition 4.1, p.246]{War71}.
Hence $e_{j}R/I$ is indecomposable. Since $R$ is Krull-Schmidt,  ${\rm End}_{R}(e_{j}R/I)$ is a local ring.

\end{proof}

Let $R$ be any ring and $M$ be any finitely presented right
$R$-module.  An Auslander-Bridger dual of $M$ is denoted by $D(M)$
and defined as follows. Choose an exact sequence $Q\overset{\phi}{\rightarrow}P\rightarrow M\rightarrow0$
in which $P$ and $Q$ are finitely generated projective right $R$-modules.
Define $D(M)$ to be the cokernel of the homomorphism $\phi^{+}:P^{+}\rightarrow Q^{+}$
where $X^{+}$=${\rm Hom}_{R}(X,R_{R})$, for any right $R$-module $X$
\cite{War75}. Although $D(M)$ depends on the choice of exact sequence, if $D^{\prime}(M)$ is another such dual of $M$ then $D(M)\oplus A\simeq D^{\prime}(M)\oplus B$ for some finitely generated
projective modules, $A,B.$

\begin{lem}\label{lem:3.12(3.12),2report,p.19}
Let $m\in \mathbb Z^{+}$  and let M be any $($1,m$)$-presented right R-module. Then   $D(M)$ is a $(n,m)$-pure-projective left
R-module, for all $n\in \mathbb Z^{+}.$

 \end{lem}

\begin{proof}   Applying Hom$_R(-,R_R)$   to a presentation  $R_{R}^{m}\overset{\lambda_H}{\longrightarrow}R_{R}^{1}\rightarrow M\rightarrow0$ of $M$ gives the presentation $_RR^{1}\overset{\rho_H}{\longrightarrow}_RR^{m}\rightarrow D(M)\rightarrow0$ of $D(M)$.   Thus $D(M)$ is $(m,1)$-presented hence $(1,m)$-pure-projective, hence $(n,m)$-pure-projective for all $n\geq 1$.

 \end{proof}

\begin{prop}\label{prop:3.12(3.12(3)),2report,p.19} Let $m\in \mathbb Z^{+}$.
Suppose that  one of the following two conditions is satisfied.

$\left(1\right)$ The  ring $R$ is   semiperfect  and   $I$ is  a nonzero ideal
 with $\textup{gen}(I_R)=m+1$ and  $I\subseteq e_{j}R$
for some local idempotent $e_j$ of R.

$\left(2\right)$ The ring $R$ is   Krull-Schmidt  and  $I$ is a nonzero right ideal
with $\textup{gen}(I)=m+1$ and  $I\subseteq e_{j}R$ for some local idempotent $e_j$ of R.

\noindent Then $D(e_{j}R/I)$ is not an $L_{\mathcal H}$-pure-projective left
R-module, where $\mathcal H$=$\bigcup$$\{$$\underset{s\times t}{M}(R)\mid s,t\in \mathbb Z^{+}$
with t<m+1$\}.$

\end{prop}

\begin{proof}
 By Lemma~\ref{lem:3.10(3.10),2report,p.18}, ${\rm End}_{R}(e_{j}R/I)$ is a local ring and hence ${\rm End}_{R}(D(e_{j}R/I))$
is  local  \cite[Theorem 2.4, p.196]{War75}.  Since {\rm gen}$(e_{j}R/I)$=$1$ and {\rm rel}$(e_{j}R/I)$= $m+1$ $($by Lemma~\ref{lem:3.10(3.10),2report,p.18}$),$  it follows easily that  {\rm gen}$(D(e_{j}R/I))$=$m+1$
and {\rm rel}$(D(e_{j}R/I))$=$1.$ Since $e_jR/I$ is not projective, neither is $D(e_{j}R/I)$ so, by Proposition~\ref{prop:3.2(3.2),2report,p.16},  $D(e_{j}R/I)$ is not $L_{\mathcal H}$-pure-projective. \end{proof}

The following theorem is a generalization of   \cite[Theorem 3.5(1), p.3888]{AlkaCo11}.

\begin{thm}\label{thm:3.13(3.13),2report,p.20} Let $(n,m)$ and $(r,s)$
be any two  pairs of positive integers such that  $n\neq r$. Suppose that  one of the following two conditions is satisfied:

$\left(a\right)$  $R$ is   semiperfect  and   there exists an  ideal
I of R with  $\textup{gen}(I_R)=$$\textup{max}\{n,r\}$  and $I\subseteq e_{j}R$
for some local idempotent $e_{j}$

$\left(b\right)$ $R$ is   Krull-Schmidt  and   there exists a right ideal
I of R with  $\textup{gen}(I)=$$\textup{max}\{n,r\}$ and $I\subseteq e_{j}R$
for some local idempotent $e_{j}$.

\noindent Then:

 $\left(1\right)$ $(m,n)$-purity and $(s,r)$-purity
of short exact sequences of left R-modules are not equivalent;

$\left(2\right)$ $(n,m)$-purity and $(r,s)$-purity
of short exact sequences of right R-modules are not equivalent.\end{thm}

\begin{proof}   $\left(1\right)$    Without loss of generality, we can assume  that $n<r$. By
 Lemma~\ref{lem:3.12(3.12),2report,p.19} and  Proposition~\ref{prop:3.12(3.12(3)),2report,p.19}, $D(e_{j}R/I)$ is $(s,r)$-pure-projective
 and  not $(m,n)$-pure-projective. Thus $(m,n)$-pure-projectivity and $(s,r)$-pure-projectivity of left $R$-modules are not equivalent and hence by  Corollary~\ref{cor:2.3(2.15(b)),p.3},  $(m,n)$-purity and $(s,r)$-purity
for left $R$-modules are not equivalent.

$\left(2\right)$ By $\left(1\right)$ and Corollary~\ref{cor:2.3(2.15(b)),p.3}.
\end{proof}

\begin{cor}\label{cor:3.14(3.14),blue,p.20(a)} Let $R$ be a local ring, let $I$ be a finitely
generated ideal  of R and set  $\textup{gen}(I_R)=r$, then for all $n<r$ and for all m,s:

 $\left(1\right)$ $(m,n)$-purity and $(s,r)$-purity
for left R-modules are not equivalent.

$\left(2\right)$  $(n,m)$-purity and $(r,s)$-purity
for right R-modules are not equivalent. \end{cor}

\begin{proof}
Since  $R$ is local it is a semiperfect and 1 is a local idempotent. By Theorem~\ref{thm:3.13(3.13),2report,p.20}, the
result holds.\end{proof}

Let $M$ be a finitely generated left module over a semiperfect
ring $R.$ Warfield in \cite{War75} defined Gen$(M)$ to be the number
of summands in a decomposition of $M/JM$ as a direct sum of simple
modules where $J=J(R).$  If $M$ is a finitely presented left module over a semiperfect ring
$R,$ and $f:P\rightarrow M$ a projective cover, with $K=$ker$(f),$
then Warfield  defined Rel$(M)$ by Rel$(M)$=Gen$(K).$ If $M$ is a left $R$-module
and $x\in M$, we say $x$ is a local element if $Rx$ is a local module. The  number of elements in any minimal generating  set of
local elements of $M$ is exactly Gen$(M)$  \cite[Lemma 1.11]{War75}. One may use these to obtain similar results, for example the following.

\begin{prop}\label{prop:3.20(3.27),blue,p.12}
Let H be a matrix over a semiperfect ring R such that $L_ H$ is not projective and \textup{End}$_{R}(L_{H})$ is a local ring and let $\mathcal H$=$\{K\mid K$ is a matrix with \textup{Gen}$(L_{H})>$ \textup{Gen}$(L_{K})$ or   \textup{Rel} $(L_{H})>$\textup{Rel} $(L_{K})\}.$
Then $L_{H}$ is not
$L_{\mathcal H}$-pure-projective.\end{prop}

\begin{proof}
Assume that $L_{H}$ is $L_{\mathcal H}$-pure-projective, thus by   Proposition~\ref{cor:1.18(2.5),2report,p.11},  $L_{H}\in$ add$(L_{\mathcal H}\cup\{_{R}R\}).$ Since End$_{R}(L_H)$ is a local ring, $L_H$ is as in   Proposition~\ref{thm:2.7(3.20),blue,p.5}, a direct summand of a module in $L_{\mathcal H}\cup\{_{R}R\}.$
Thus either $L_H$ is a direct summand of $L_{D},$
where $D\in\mathcal{H}$ or $L_H$ is a direct
summand of $_{R}R$.  Since $L_H$ is not projective,  $L_H$ is a direct summand of $L_{D},$
thus by \cite[Lemma 1.10, p.192]{War75}, Gen$(L_H)\leq$Gen$(L_{D})$
and Rel$(L_H)\leq$Rel$(L_{D})$ and this contradicts
${D}\in\mathcal{H}$.  Therefore,
$L_{H}$ is not $L_{\mathcal{H}}$-pure-projective. \end{proof}

\begin{rem*}\label{cor:3.21(3.28),blue,p.13}
Since, if $K$ is an $r\times q$ matrix, we have \textup{Gen}$(R).q\geq$\textup{Gen}$(R).$\textup{gen}$($ $L_K)\geq \textup{Gen}(L_K)$ and similarly for relations, if $H$ is as in Proposition~\ref{prop:3.20(3.27),blue,p.12} then $L_H$ is not  $L_{{\mathcal H}_{i}}$-pure-projective  for any of the sets of matrices:

${\mathcal H}_{1}$=$\{K\mid$ \textup{Gen}$(L_{H})>$\textup{Gen}$(R)$ \textup{gen}$(L_{K})$ or \textup{Rel}$(L_{H})>$\textup{Gen}$(R)$ \textup{rel}$(L_{K})\};$

${\mathcal H}_{2}$=$\{\underset{r\times q}{K}\mid r,q\in \mathbb Z^{+}$such that
\textup{Gen}$(L_{H})>$$q$\textup{Gen}$(R)$ or \textup{Rel}$(L_{H})>$\textup{Rel}$(L_{K})\};$

${\mathcal H}_{3}$ =$\bigcup\{\underset{r\times q}{M}(R)\mid r,q\in \mathbb Z^{+}$such that
\textup{Gen}$(L_{H})>$$q$\textup{Gen}$(R)$ or \textup{Rel}$(L_{H})>$$r$\textup{Gen}$(R)\}.$
\end{rem*}

\section{Purity over finite-dimensional algebras}

\subparagraph*{\textmd{In this section we assume some knowledge of the representation theory of finite-dimensional algebras, for which see   \cite{AsSiSk06}, \cite{AuReSm95} for example. Let $R$ be a Krull-Schmidt ring and let $M$ be any finitely presented left $R$-module. We will use ind$(M)$ to denote the class of $($isomorphism types of$)$ indecomposable direct summands of $M$. If $S$ is a class of finitely presented left $R$-modules, we  define ind$(S)=\underset{M\in S}{\bigcup}$ind$(M).$ }}

\begin{prop}\label{prop:4.1(4.12)(a),p.7(a)}
 Let R be a Krull-Schmidt ring
and let S be a class of finitely presented left R-modules. Then the
following statements are equivalent for a left R-module M:

$\left(1\right)$ M is S-pure-projective.

$\left(2\right)$ M is  \textup{ind}$(S)$-pure-projective.

$\left(3\right)$ M isomorphic to a direct sum of modules in \textup{ind}$(S\cup\{_{R}R\}).$ \end{prop}

\begin{proof}  Since $R$ is a Krull-Schmidt ring, each element in $S\cup\{_{R}R\}$
is a direct sum of modules in ind$(S\cup$ $\{_{R}R\})$ so this follows by  Proposition~\ref{cor:1.18(2.5),2report,p.11}.\end{proof}

The following corollary is immediate from Proposition~\ref{prop:4.1(4.12)(a),p.7(a)} and Theorem~\ref{thm:2.1(2.15),(2.17),cor(2.18),cor(1.6),2report,p.14-16,p.5}.

\begin{cor}\label{cor:4.2(4.13),p.8}
 Let R be a Krull-Schmidt ring and let S
and T be two classes of finitely presented left R-modules. Then
T-purity implies S-purity if and only if  $\textup{ind}(S)\subseteq \textup{ind}(T\cup\{_{R}R\}).$
 \end{cor}

Let $R=k\tilde{A}_{1}$ be the Kronecker algebra
over an algebraically closed field $k.$  Left $R$-modules may be viewed as representations of the quiver $\underset{\beta}{\overset{\alpha}{_{1}\bullet\leftleftarrows\bullet_{2}}}$. The preinjective  and preprojective indecomposable
finite-dimensional left $R$-modules are up to isomorphism uniquely
determined by their dimension vectors. For $n\in\mathbb{N}$
we will denote by $I_{n}$ (resp. $P_{n}$)
the finite-dimensional indecomposable preinjective (resp. preprojective) left $R$-module with  dimension vector $(n,n+1)$ (resp. $(n+1,n)).$  Also, for $n\in \mathbb{Z}^{+}$ we will use   $R_{\lambda,n}$ to denote the finite-dimensional indecomposable regular left $R$-module with dimension vector $(n,n)$   and parameter  $\lambda\in k\cup\{\infty\}$   where  $R_{\lambda,1}$  is the module $\underset{\lambda}{\overset{1}{k\leftleftarrows k}}$ for $\lambda \in k$ and  $R_{\infty,1}=\underset{1}{\overset{0}{k\leftleftarrows k}}$.

\begin{example}\label{example:4.3((Example 4.14, p.8) and (New version of Proposition 4.32, p.27))}
 Let $R=k\tilde{A_{1}}$ be the Kronecker algebra
over an algebraically closed field $k.$  Let $n,r\in\mathbb{Z}^{+}$ and
let $S_{1}=\{P_{i}\mid i\leq n\},$ $S_{2}=\{I_{i}\mid i\leq n-1\},$
$S_{3}=\{R_{\lambda,i}\mid i\leq n$ and $\lambda\in k\cup\{\infty\}\}$
and $S_{4}=\{R_{\lambda,1}\mid\lambda\in k\cup\{\infty\}\}\cup\{P_{0},P_{1}\}.$
Then:

$\left(i\right)$ $S_{1}\cup S_{2}\cup S_{3}$-purity$=(\aleph_{0},n)$-purity$=(2n+1,n)$-purity,
for short exact sequences of left $R$-modules.

$\left(ii\right)$ $S_{4}$-purity$=(1,1)$-purity for  left R-modules.

$\left(iii\right)$  $\,$  $(1,1)$-purity is not equivalent to $(\aleph_{0},n)$-purity
for  left $R$-modules.

$\left(iv\right)$  $($n,1$)$-purity is not equivalent to $($r,2$)$-purity for  left $R$-modules.

 \end{example}

\begin{proof}
$\left(i\right)$ Let $\mathcal H =\underset{m\in\mathbb{Z}^{+}}{\bigcup}\underset{m\times n}{M}(R)$.  It follows directly from the description of the finite-dimensional indecomposable modules and Remark~\ref{rem:3.1(3.1),2report,p.16} that   ind$(L_{\mathcal H})=S_1\cup S_2\cup S_3$.  Thus, by Proposition~\ref{prop:4.1(4.12)(a),p.7(a)} we have that  $S_1\cup S_2\cup S_3$-purity=$L_{\mathcal H}$-purity=$(\aleph_{0},n)$-purity for left $R$-modules.

Let $M\in S_{1}\cup S_{2}\cup S_{3}.$  It can be  checked that, if $M\in S_1$ then rel$(M)\leq 2n-1$, if $M\in S_2$ then rel$(M)\leq 2n+1$ and if $M\in S_3$ then rel$(M)\leq 2n$ and hence rel$(M)\leq 2n+1$ in all cases. Since gen$(M)\leq n,$ each module in  $S_{1}\cup S_{2}\cup S_{3}$ is $(n,2n+1)$-presented. Thus  $(2n+1,n)$-purity=$S_{1}\cup S_{2}\cup S_{3}$-purity$=(\aleph_{0},n)$-purity.

$\left(ii\right)$ Let $\lambda\in k\cup \{\infty\}$   and let $M\in R_{\lambda,1}\oplus P_0$.  Since the sequence $_{R}R\overset{(\alpha+\lambda\beta)\times-}{\longrightarrow}_{R}R\rightarrow M\rightarrow0$  is exact, $M$ is $(1,1)$-presented  and hence $R_{\lambda,1}$  is a direct summand of a $(1,1)$-presented module. Thus every module in $S_4$ is a direct summand of a $(1,1)$-presented module.  Conversely, let  $N$ be any indecomposable direct summand of a $(1,1)$-presented left $R$-module, thus gen$(N)=1$  and rel$(N)\leq 2$ $($by Remark~\ref{rem:3.1(3.1),2report,p.16}$)$ and hence either $N=P_0$ or $N=P_1$   or $N=R_{\lambda,1}$  for some $\lambda\in k\cup \{\infty\}$. Thus $N$ is a direct summand of a module in $S_4\cup \{_RR\}$. By Proposition~\ref{thm:2.7(3.20),blue,p.5},   $S_{4}$-purity$=(1,1)$-purity.

$\left(iii\right)$  Assume that $(1,1)$-purity=$(\aleph_{0},n)$-purity for some $n\in\mathbb{Z}^{+}$. Thus, by $\left(i\right)$ and $\left(ii\right)$ above we have that $S_{4}$-purity=$S_{1}\cup S_{2}\cup S_{3}$-purity. This contradicts  Corollary~\ref{cor:4.2(4.13),p.8}, because $I_0 \in S_{1}\cup S_{2}\cup S_{3}$ and $I_0 \notin S_4$.

$\left(iv\right)$ Note that  $R_{R}=e_{1}R \oplus e_{2}R$, where $e_1R$ $($resp. $e_2R$$)$ is the preprojective right $R$-module of dimension vector $(0,1)$ $($resp. $(1,2)$$)$.  Let $I_R=J(e_2R)$, since $I_R=\alpha R\oplus\beta R$   it follows that gen$(I_R)=2.$  By Theorem~\ref{thm:3.13(3.13),2report,p.20} we have that $(n,1)$-purity and $(r,2)$-purity for left $R$-modules are not equivalent.

\end{proof}

\begin{prop}\label{prop:4.4(new)}  Let $R$ be a finite-dimensional algebra over a field $k$. If $R$ is  not of finite representation type, then for every $r\in \mathbb{Z}^{+}$, there is $n>r$ such that  $(\aleph_{0},n)$-purity$\neq (\aleph_{0},r)$-purity for  left R-modules.
\end{prop}

\begin{proof} Suppose that $R$ is  not of finite representation type. Assume that there is $r\in \mathbb{Z}^{+}$ such that for all $n>r$  then $(\aleph_{0},n)$-purity$= (\aleph_{0},r)$-purity for left $R$-modules. Since $R$ is a finite-dimensional algebra and it is not finite representation type it follows from \cite[Corollary 1.5, p.194]{AuReSm95} that there is a finitely generated indecomposable left $R$-module $M$ such that gen$(M)\geq r+1$.   By assumption, $(\aleph_{0},\textup{gen}(M))$-purity$= (\aleph_{0},r)$-purity for left $R$-modules and hence by Corollary~\ref{cor:4.2(4.13),p.8},  $M\in$ind$(\{(r,s)$- presented left $R$-modules $\mid$ $s\in \mathbb{Z}^{+}\})$, which  is a contradiction.

\end{proof}

Let $R$ be an algebra over a field $k$. From now, we use $M^{*}$ to denote Hom$_k(M,k)$ for any $R$-module $M$.

\begin{prop}\label{prop:4.4(Proposition 4.36, p.33 and Proposition 4.48(a), p.43(a))} Let R be a finite-dimensional algebra over a field $k$ and let $\mathcal H$ be a set of matrices
over R. Then a left R-module
M  is $L_{\mathcal H}$-pure-injective if and only if  M is  a direct summand of a direct product
of modules in $\textup{ind}(\{$$D_{H}^{*}\mid H\in\mathcal H\cup\{\underset{1\times1}{0}\}\}).$
\end{prop}

\begin{proof}
 This follows by  Theorem~\ref{thm:28(4.27),black,p.22} since each module $D_{H}^{*}$  is a finite direct sum of indecomposable modules.
\end{proof}

We now describe these modules in terms of ind$(\{L_H\mid H\in\mathcal H\cup\{\underset{1\times1}{0}\}\})$.

\begin{thm}\label{thm:4.5} Let R be a finite-dimensional algebra over a field $k$ and let $S$ be a set of indecomposable finite-dimensional modules. Then the $S$-pure-injective left $R$-modules are the  direct summands of  direct products
of modules in $\tau S\cup R$-\textup{inj},  where $\tau$  is the Auslander-Reiten translate and R-\textup{inj} denotes the set of indecomposable injective left R-modules.

\end{thm}

\begin{proof} The Auslander-Reiten translate of a module $M$ is given by the formula $\tau M=(DM)^*$ where $DM$ is the Auslander-Bridger  dual $(=$transpose$)$ of $M$ obtained from a minimal projective resolution of $M$. In particular $\tau L_H=(D_H)^*$ so this follows from Proposition~\ref{prop:4.4(Proposition 4.36, p.33 and Proposition 4.48(a), p.43(a))}.
\end{proof}

\begin{cor}\label{cor:4.5(Proposition 4.48, p.43)}
 Let R be  a finite-dimensional algebra over a field  $k,$
let $\mathcal H$ be a set of matrices over R. If  $\textup{ind}\{$\textup{Hom}$_{k}(D_{H},k)\mid H\in\mathcal H\cup\{\underset{1\times1}{0}\}\}$  is finite then it is the set of  indecomposable $L_{\mathcal H}$-pure-injective
left R-modules and every $L_{\mathcal H}$-pure-injective module is a direct sum of copies of these modules.\end{cor}

\begin{proof} This follows since if $M$ is indecomposable of finite length over its endomorphism ring then every product of copies of $M$ is a direct sum of copies of $M$ $($see, for example, \cite[Theorem 4.4.28, p.180]{Pre09}$)$.

\end{proof}

Recall $($see\cite[3.4.7]{Pre09}$)$ that a subclass $T$ of $R$-Mod is said to be
definable if it is closed under direct products, direct limits and
pure submodules. A class $T$ of pure-injective
modules closed under direct products, direct summands and isomorphisms is definable if and only if each direct sum of modules in $T$ is pure-injective, that is if and only if each element in $T$ is $\Sigma$-pure-injective $($see, for example, \cite[ 4.4.12]{Pre09}$)$. In this case every module in $T$ is a direct sum of indecomposable modules.
Let $S$ be a class of finitely presented left $R$-modules. We denote by $S$-Pinj the
class of  $S$-pure-injective
left $R$-modules.

\begin{cor}\label{cor:4.7(Lemma 4.44, p.40)}
Let R be a finite-dimensional algebra
over a field k and let $\mathcal H$ be a set of matrices over R. Then $L_{\mathcal H}$-Pinj is a definable subclass of $R$-Mod if and only if each direct sum of modules in $\textup{ind}(\{D_{H}^{*}\mid H\in\mathcal H\})$ is pure-injective. \end{cor}

\begin{proof} Let $T=\textup{ind}(\{D_{H}^{*}\mid H\in\mathcal H\})$ and $T^\prime=T\cup R$-inj.

One direction follows from the remarks above and  Proposition~\ref{prop:4.4(Proposition 4.36, p.33 and Proposition 4.48(a), p.43(a))}.

$(\Leftarrow)$.  By hypothesis, each direct sum of modules in $T$ is pure-injective. Since $R$ is a left Noetherian ring, each direct sum of modules in $R$-inj is injective.  Thus each direct sum of modules in $T^\prime$ is pure-injective and hence  is $\Sigma$-pure-injective. Let $M\in L_{\mathcal H}$-Pinj. By Proposition~\ref{prop:4.4(Proposition 4.36, p.33 and Proposition 4.48(a), p.43(a))}, there exists a subfamily $\{M_i\}_{i\in I}$ of $T^\prime$ such that $M$ is  a direct summand of $\underset{i\in I}{\prod}M_{i}$. By the proof above, $\underset{i\in I}{\bigoplus}M_{i}$ is $\Sigma$-pure-injective. Since  $\underset{i\in I}{\prod}M_{i}$ is in the definable subcategory generated by $\underset{i\in I}{\bigoplus}M_{i}$ it follows from \cite[Proposition(4.4.12), p.176]{Pre09} that $\underset{i\in I}{\prod}M_{i}$ is $\Sigma$-pure-injective. It follows that  $M$ is $\Sigma$-pure-injective and hence each element in $L_{\mathcal H}$-Pinj is $\Sigma$-pure-injective. Therefore  $L_{\mathcal H}$-Pinj is a definable subclass of $R$-Mod.

\end{proof}

Every finite-dimensional module is $\Sigma$-pure-injective and by \cite[Theorem 4.6, p.750]{Len83} every direct sum of preinjective  modules is $\Sigma$-pure-injective. The equivalence of (1) and (2)  in the next result therefore follows from the description  of the $\Sigma$-pure-injective modules in  \cite[Theorem 2.1, p.847]{PrPu02}  and the equivalence with (3) follows since the duality Hom$_k(-,k)$ interchanges preprojective and preinjective modules and sends regular modules to regular modules.

\begin{prop}\label{prop:4.8(prop 4.44(a), p.41)}
Let $R$ be a tame hereditary finite-dimensional
algebra over a field $k$ and let $\mathcal H$ be a set of matrices over
$R.$ Then the following statements are equivalent.

$\left(1\right)$   $L_{\mathcal H}$-Pinj is a definable subclass of $R$-Mod.

$\left(2\right)$ The set of preprojective or regular modules in $\textup{ind}(\{D_{H}^{*}\mid H\in\mathcal H\})$ is finite.

$\left(3\right)$ The set of preinjective or regular modules in $\textup{ind}(\{D_{H}\mid H\in\mathcal H\})$ is finite.
\end{prop}

Let  $_{R}$pinj be  the isomorphism classes of  indecomposable pure-injective
left $R$-modules and let  $T\subseteq R$-ind, the class of all finitely presented indecomposable
left $R$-modules. We use fsc$(T)$ $($resp. $\overline{T})$ to
denote the closure of $T$ in the full support topology $($resp.
the Ziegler topology$).$ Recall that fsc$(T)$ is the class Prod$(T)\cap_{R}$pinj.
 See for instance \cite [Sections 5.1.1 and 5.3.7]{Pre09} for details
about the Ziegler topology and the full support topology.

\begin{prop}\label{prop:4.10(prop 4.49, p.44)}
Let R be a tame hereditary
finite-dimensional algebra over a field k. Let $\mathcal H$ be a set of
matrices over R such that $L_{\mathcal H}$-Pinj is  definable.
Then the following statements are equivalent.

$\left(1\right)$ The generic module is $L_{\mathcal H}$-pure-injective.

$\left(2\right)$ The set of preinjective left R-modules in  $\textup{ind}(\{D_{H}^{*}\mid H\in\mathcal H\})$  is
infinite.

$\left(3\right)$ All but at most $n(R)-2$ Pr\"{u}fer modules are $L_{\mathcal H}$-pure-injective,
where $n(R)$ is the number of isomorphism classes of simple R-modules.

$\left(4\right)$  At least one Pr\"{u}fer R-module is $L_{\mathcal H}$-pure-injective. \end{prop}

\begin{proof}
$\left(1\right)\Rightarrow\left(2\right).$ Let $T=\textup{ind}($ $\{D_{H}^{*}\mid H\in\mathcal H\})$. Assume that the set of preinjective left $R$-modules in $T$ is finite. Since $L_{\mathcal H}$-Pinj is a definable it follows from Proposition~\ref{prop:4.8(prop 4.44(a), p.41)} that $T$ is finite.  By Corollary~\ref{cor:4.5(Proposition 4.48, p.43)}, the generic module cannot be $L_{\mathcal H}$-pure-injective.

$\left(2\right)\Rightarrow\left(3\right).$ Let $X$ be the class
of all indecomposable $L_{\mathcal H}$-pure-injective modules. Since $L_{\mathcal H}$-Pinj is definable it follows from \cite[Theorem 5.1.1, p.211]{Pre09} that $X$ is  a closed set of the Ziegler topology. Since $X$  contains infinitely many non-isomorphic preinjective modules, by \cite[Corollary, p.113]{Rin98}, all but
at most $n(R)-2$ Pr\"{u}fer modules belong to $X,$ where $n(R)$
is the number of isomorphism classes of simple $R$-modules.

$\left(3\right)\Rightarrow\left(4\right).$ This is obvious.

$\left(4\right)\Rightarrow\left(1\right).$ Assume that there is a
Pr\"{u}fer  module which is $L_{\mathcal H}$-pure-injective.
As noted in $\left(3\right)$,    $X$
is a closed set of the Ziegler topology and by hypothesis, it contains at least one module  which is
not of finite length. By \cite[Theorem, p.106]{Rin98}, the generic
module belongs to $X$.\end{proof}

\begin{rem}\label{rem:4.11(Corollary 4.50, p.45)}
If  R is the Kronecker
algebra over a field k then condition
$\left(3\right)$ above becomes:  $\left(3\right)$ Every Pr\"{u}fer module is $L_{\mathcal H}$-pure-injective.

\end{rem}

\begin{lem}\label{lem:4.12}
Let $T\subseteq R$-ind. If \emph{Prod}$(T)$ is definable then $\overline{T}=$\emph{fsc}$(T).$\end{lem}
\begin{proof}
Suppose that Prod$(T)$ is definable. It is clear that fsc$(T)\subseteq\overline{T}.$
Since $T\subseteq$Prod$(T)$ it follows that $D(T)\subseteq$Prod$(T),$
where $D(T)$ is the definable subcategory generated by $T.$ Thus
$\overline{T}\subseteq$fsc$(T)$ and hence $\overline{T}=$fsc$(T).$\end{proof}

\begin{rem}\label{rem:4.13}
Let $T$ be a class of pure-injective left $R$-modules and let $S\subseteq T.$
If \emph{Prod}$(T)$ is a definable subclass of $R$-Mod then so is \emph{Prod}$(S).$\end{rem}

\begin{cor}\label{cor:4.14}
Let R be a tame hereditary finite-dimensional algebra over a field
k and let $\textbf{I}_{1}$ be a class of indecomposable preinjective left
R-modules. Then \emph{fsc}$(\textbf{I}_{1})=\overline{\textbf{I}_{1}}.$ \end{cor}

\begin{proof}
By \cite[Theorem 3.2, p.351]{Zay97}, Prod$(\emph{\textbf{I}})$ is
definable, where $\emph{\textbf{I}}$  is the class of
all indecomposable preinjective left $R$-modules. Since $\emph{\textbf{I}}_{1}\subseteq \emph{\textbf{I}}$
it follows from     Remark~\ref{rem:4.13}   that Prod$(\emph{\textbf{I}}_{1})$
is definable. By  Lemma~\ref{lem:4.12}, fsc$(\textbf{I}_{1})=\overline{\textbf{I}_{1}}.$\end{proof}

\begin{rem}\label{rem:4.15}
Let $R$ be a  finite-dimensional algebra over a field
$k$ and let $T\subseteq R$-\emph{ind}. Then $T$ is the class of all indecomposable
finite-dimensional left $R$-modules in \emph{Prod}$(T).$ This follows from \cite[Corollary 5.3.33, p.250]{Pre09}
\end{rem}

The following fact is known; it can be found stated in \cite[p.47]{Rin00}.
We include a proof here.

\begin{prop}\label{prop:4.16(Proposition 6.5, p.4)}
 Let R be a tame hereditary finite-dimensional
algebra over a field $k$ and let $\textbf{P}_{1}$ be a class of indecomposable preprojective left
R-modules. Then \emph{fsc}$(\textbf{P}_{1})=\textbf{P}_{1}.$  \end{prop}

\begin{proof}

Let $M\in$fsc$(\emph{\textbf{P}}_{1})$.  Thus $M$ is a direct summand of $\underset{i\in I}{\prod}P_{i}$
where $P_{i}\in \emph{\textbf{P}}_{1}.$ Choose a non-zero element $a\in M,$ so $a_{j}\neq0$
for some $j\in I,$ where $a_{j}$ is the jth component in $a.$ Define
$\alpha:M\rightarrow P_{j}$ by $\alpha=\pi_{j}i$ where $i:M\rightarrow\underset{i\in I}{\prod}P_{i}$
is the inclusion and $\pi_{j}:\underset{i\in I}{\prod}P_{i}\rightarrow P_{j}$
is the projection. Since $\alpha(a)=a_{j}\neq0$ it follows that Hom$_{R}(M,\emph{\textbf{P}})\neq0,$
where $\emph{\textbf{P}}$ is the class of all indecomposable preprojective left $R$-modules.
By \cite[Lemma 1, p.46]{Cra98}, $M$ has a preprojective direct summand,
and  hence $M$ is  finite-dimensional  and therefore we have from Remark~\ref{rem:4.15} that $M\in \emph{\textbf{P}}_{1}.$

\end{proof}

\begin{lem}\label{lem:4.17(Lemma 6.9, p.7)} Let R be a tame hereditary finite-dimensional algebra
over a field k and let $\textbf{R}_{1}$ be a class of indecomposable regular
left R-modules. Then:

$\left(1\right)$ The generic module does not belong to \emph{fsc}$(\textbf{R}_{1})$.

$\left(2\right)$ There is no Pr\"{u}fer  R-module in \emph{fsc}$(\textbf{R}_{1})$.\end{lem}
\begin{proof}
$\left(1\right)$ Assume that the generic module $G\in$fsc$(\emph{\textbf{R}}_{1}),$
thus $G\in$Prod$(\emph{\textbf{R}}_{1}).$ As in the proof of Proposition~\ref{prop:4.16(Proposition 6.5, p.4)} it follows that  Hom$_{R}(G,\emph{\textbf{R}}_{1})\neq 0,$ contradicting
\cite[p.46]{Rin00}. Therefore  $G\notin$fsc$(\emph{\textbf{R}}_{1})$.

$\left(2\right)$ Assume that there is a Pr\"{u}fer module $M$
such that $M\in$fsc$(\emph{\textbf{R}}_{1}).$  By \cite[Proposition 3, p.110]{Rin98},
the generic module $G$ is a direct summand of $M^{I}$ for some $I$ so
$G\in$Prod$(\emph{\textbf{R}}_{1})$ and this contradicts $\left(1\right)$ above. Thus there is no Pr\"{u}fer  module in fsc$(\emph{\textbf{R}}_{1})$. \end{proof}

Let $R$ be a tame hereditary finite-dimensional algebra over a field
$k$ and let $S$ be a simple regular left $R$-module $($that is, a module which is simple in the category of regular modules$)$. We use $S[\infty]$
$($resp. $\hat{S})$ to denote the Pr\"{u}fer $($resp. adic$)$ left
$R$-module corresponding to $S,$ see \cite[p.106]{Rin98} for the definitions of these modules. Also, we use $T_{S}$ to denote the
class $T_{S}=\{M\mid M$ is an indecomposable regular left $R$-module
with Hom$_{R}(M,S)\neq0\}.$

\begin{thm}\label{thm:4.18(Theorem 6.10, p.8)}
 Let $R$ be a tame hereditary finite-dimensional
algebra over a field $k.$ Let $\textbf{R}_{1}$ be a class of indecomposable
regular left R-modules and let $S$ be a simple regular left $R$-module.
Then the following statements are equivalent:

$\left(1\right)$ $\hat{S}\in$\emph{fsc}$(\textbf{R}_{1}).$

$\left(2\right)$ $\textbf{R}_{1}\cap T_{S}$ is infinite.

$\left(3\right)$ $\hat{S}\in$\emph{fsc}$(\textbf{R}_{1}\cap T_{S}).$ \end{thm}

\begin{proof}
$\left(1\right)\Rightarrow\left(2\right).$ Suppose that $\hat{S}\in$fsc$(\emph{\textbf{R}}_{1}).$
Assume that $\emph{\textbf{R}}_{1}\cap T_{S}$ is  finite. Let $D=\{M\mid$Hom$_{R}(M,S)=0\}.$
By \cite[Examples, p.42]{Cra98}, $D$ is a definable subclass of
$R$-Mod and hence $C=D\cap_{R}$pinj is a closed set in the Ziegler
topology. Since $\emph{\textbf{R}}_{1}\cap T_{S}$ is a finite class of finite-dimensional
indecomposable modules it follows from \cite[ 2.5]{Cra98} that $\emph{\textbf{R}}_{1}\cap T_{S}$
is a closed set in the Ziegler topology and hence $C\cup(\emph{\textbf{R}}_{1}\cap T_{S})$
is. Thus $C\cup(\emph{\textbf{R}}_{1}\cap T_{S})$ is a closed set in the full
support topology. Since $\emph{\textbf{R}}_{1}\subseteq C\cup(\emph{\textbf{R}}_{1}\cap T_{S})$
it follows that fsc$(\emph{\textbf{R}}_{1})\subseteq$fsc$(C\cup(\emph{\textbf{R}}_{1}\cap T_{S}))=C\cup(\emph{\textbf{R}}_{1}\cap T_{S}).$
Since Hom$_{R}(\hat{S},S)\neq0$ it follows that $\hat{S}\notin C\cup(\emph{\textbf{R}}_{1}\cap T_{S})$
and hence $\hat{S}\notin$fsc$(\emph{\textbf{R}}_{1})$ and this contradicts the hypothesis.
Thus $\emph{\textbf{R}}_{1}\cap T_{S}$ is infinite.

$\left(2\right)\Rightarrow\left(3\right)$ Suppose that $\emph{\textbf{R}}_{1}\cap T_{S}$
is infinite, thus $(\emph{\textbf{R}}_{1}\cap T_{S})^{*}$ is an infinite class
of regular right $R$-modules.
Let $X\in(\emph{\textbf{R}}_{1}\cap T_{S})^{*},$ thus $X=M^{*}$ for some $M\in \emph{\textbf{R}}_{1}\cap T_{S}.$
 Hence Hom$_{R}(M,S)\neq0.$ Thus  Hom$_{R}(S^{*},X)\neq0$ for all $X\in(\emph{\textbf{R}}_{1}\cap T_{S})^{*}.$
By \cite[Proposition 1, p.107]{Rin98}, $S^{*}[\infty]$ is the direct
limit of a chain of monomorphisms $X_{1}\rightarrow X_{2}\rightarrow X_{3}\rightarrow\cdots$
with $X_{i}\in(\emph{\textbf{R}}_{1}\cap T_{S})^{*}.$ Therefore, by \cite[Proposition 2.1, p.736]{Len83},
there is a pure exact sequence $0\rightarrow N\rightarrow\underset{j\in J}{\bigoplus}Y_{j}\rightarrow S^{*}[\infty]\rightarrow0$
with $Y_{j}\in(\emph{\textbf{R}}_{1}\cap T_{S})^{*}.$ Therefore
the exact sequence $0\rightarrow(S^{*}[\infty])^{*}\rightarrow(\underset{j\in J}{\bigoplus}Y_{j})^{*}\rightarrow N^{*}\rightarrow0$
is split. By \cite[Examples, p.44]{Cra98}, $\hat{S}=(S^{*}[\infty])^{*}$ and hence $\hat{S}$ is a direct summand
of $\underset{j\in J}{\prod}N_{j}^{*}.$ Thus $\hat{S}\in$Prod$(\emph{\textbf{R}}_{1}\cap T_{S})$
and this implies that $\hat{S}\in$fsc$(\emph{\textbf{R}}_{1}\cap T_{S}).$

$\left(3\right)\Rightarrow\left(1\right)$. This is obvious.
\end{proof}

\begin{cor}\label{cor:4.19(Corollary (6.11), p.10)}   Let $R$ be a tame hereditary finite-dimensional
algebra over a field $k$ and let $\textbf{R}_{1}$ be a class of indecomposable
regular left R-modules. Then \emph{fsc}$(\textbf{R}_{1})=\textbf{R}_{1}\cup\{\hat{S}\mid \textbf{R}_{1}\cap T_{S}$
is infinite$\}.$ \end{cor}

\begin{proof}
This follows  by Theorem~\ref{thm:4.18(Theorem 6.10, p.8)},   Remark~\ref{rem:4.15} and  Lemma~\ref{lem:4.17(Lemma 6.9, p.7)}.  \end{proof}

In the following corollary we give a  complete description of the closure
of any subclass of $R$-ind in the full support topology and hence, by Theorem~\ref{thm:4.5}, a description of the indecomposable $S$-pure-injective modules for any purity defined by a class $S$ of finitely presented modules.

\begin{cor}\label{cor:4.20(Corollary (6.12), p.10)}
 Let $R$ be a tame hereditary finite-dimensional
algebra over a field $k.$ Let $\textbf{I}_{1}$ $($resp. $\textbf{P}_{1},$ resp.
$\textbf{R}_{1})$ be a class of indecomposable preinjective $($resp. preprojective,
resp. regular$)$ left R-modules. Then \emph{fsc}$(\textbf{I}_{1}\cup \textbf{P}_{1}\cup \textbf{R}_{1})=\overline{\textbf{I}_{1}}\cup \textbf{P}_{1}\cup \textbf{R}_{1}\cup\{\hat{S}\mid \textbf{R}_{1}\cap T_{S}$
is infinite$\}.$\end{cor}
\begin{proof}
This follows from Corollary~\ref{cor:4.14}, Proposition~\ref{prop:4.16(Proposition 6.5, p.4)}
and Corollary~\ref{cor:4.19(Corollary (6.11), p.10)}.\end{proof}

\section{Rings whose indecomposable modules are $S$-pure-projective.}

\subparagraph*{\textmd{Let $T$ be a set. A family $F$ of subsets of $T$ is said
to be directed if for any $U,V\in F,$ there exists $W\in F$ such
that $U\subseteq W$ and $V\subseteq W.$ }}

\subparagraph*{\textmd{By using Theorem~\ref{thm:(1.3),2report,p3}, we can prove the following lemma.}}

\begin{lem}\label{lem:5.1(Lemma 4.16, p.10)}
 Let S  be a class of finitely presented left R-modules  and
let $\{N_{i}\}_{i\in I}$ be any directed family of $S$-pure
submodules of a left R-module M. Then $N=\underset{i\in I}{\bigcup}N_{i}$
is an  $S$-pure  submodule of M.\end{lem}

Let $N$  be a submodule of a left $R$-module $M$
and let $T$ be a set of submodules of $M.$ We will use $N(T)$ to
denote the submodule $N(T)=N+\underset{A\in T}{\sum}A.$ The next lemma follows using Lemma~\ref{lem:5.1(Lemma 4.16, p.10)}.

\begin{lem}\label{lem:5.2(Lemma 4.18, p.12)}
 Let S  be a class of finitely presented left R-modules,
let $N$ be a submodule of a left $R$-module $M$ and let $T$ be
a set of submodules of $M.$ If $N(F)$ is an S-pure
submodule of M for all finite subsets F of T, then $N(T)$ is an S-pure
submodule of M. \end{lem}

\begin{defn}\label{defn:5.3(Definition 4.19, p.13)}
 Let   S  be a class of finitely presented left R-modules,  let $N$ be a submodule of a left $R$-module $M$ and let $T_{0}$
be the set of all indecomposable submodules of $M.$ A subclass $T\subseteq T_{0}$
said to be $S$-$N$-independent $($in $M)$ if $N(T)=N\oplus(\underset{B\in T}{\sum}\oplus B)$
and $N(T)$ is S-pure submodule of $M.$  This will be the case if and only if every finite subset of $T$ is $S$-$N$-independent $M$. \end{defn}

\begin{thm}\label{thm:5.5(Theorem 4.22, p.16)}
 Let S be any set of finitely presented left R-modules
and let M be a left R-module. Suppose that every S-pure submodule
$M_{0}$ of M for which $M/M_{0}$ is  indecomposable is a direct
summand of M. Then every S-pure  submodule of M is a direct
summand of M and M is a direct sum of indecomposable submodules. \end{thm}

\begin{proof}$($The following proof is based on an argument in \cite[Proposition 1.13, p.53]{Azu92}$).$
Let $N$ be any $S$-pure submodule of $M.$ If $N=M$ then $N$ is
a direct summand of $M.$ Assume that $N\neq M,$ thus there is $x\in M$\textbackslash{}$N.$
Let $F=\{K\mid N\subseteq K,$ $x\notin K$ and $K$ is an $S$-pure
submodule of $M\}.$ Since $N\in F$ it follows that $F$ is a non-empty
family. Let $\{M_{i}\}_{i\in I}$ be any directed subfamily of $F$
and let $A=\underset{i\in I}{\bigcup}M_{i}.$ It is clear that $N\subseteq A$
and $x\notin A.$ By Lemma~\ref{lem:5.1(Lemma 4.16, p.10)}, $A$ is an $S$-pure
submodule of $M$ and hence $A\in F.$ By Zorn's lemma, $F$ has a
maximal element, say $M_{0},$ thus $M_{0}$ is an $S$-pure submodule
of $M$ with $N\subseteq M_{0}$ and $x\notin M_{0}.$ We will prove
that $M/M_{0}$ is  indecomposable.

Assume that $M/M_{0}$ is not
indecomposable, thus there are two non-zero submodules $M_{1}/$ $M_{0},$
$M_{2}/M_{0}$ of $M/M_{0}$ such that $M/M_{0}=(M_{1}/M_{0})\oplus(M_{2}/M_{0}).$
  Therefore  $M_{0}\underset{\neq}{\subset}M_{1}$, $M_{0}\underset{\neq}{\subset}M_{2}$ and $M_{1}\cap M_{2}=M_{0}.$
Since $M_{1}/M_{0}$ and $M_{2}/M_{0}$ are direct summands of $M/M_{0}$
 they  are $S$-pure submodules
of $M/M_{0}.$ Since $M_{0}$ is an $S$-pure submodule of $M$ it
follows from \cite[33.3(4), p.276]{Wis91} that $M_{1}$ and $M_{2}$
are $S$-pure submodules of $M.$ Thus, by maximality of $M_{0},$
we have that $x\in M_{1}\cap M_{2}$ and this is a contradiction.

Hence $M/M_{0}$ is a non-zero indecomposable left $R$-module. By
assumption, $M_{0}$ is a direct summand of $M,$ say  $M=N_{0}\oplus M_{0}.$ Thus
$N_{0}\simeq M/M_{0}$  is a non-zero indecomposable
submodule of $M$ with   $N+N_{0}=N\oplus N_{0}.$
Since $N$ is an $S$-pure submodule of $M$ and $N\subseteq M_{0}\subseteq M$
it follows that $N$ is $S$-pure submodule of $M_{0}$
and hence $N\oplus N_{0}$ is an $S$-pure submodule of $N_{0}\oplus M_{0}=M.$
Thus, for any proper $S$-pure submodule $N$ of $M,$ there exists
a non-zero indecomposable submodule $N_{0}$ of $M$ such that $N\cap N_{0}=0$
 and $N\oplus N_{0}$ is an $S$-pure submodule
of $M.$

Let $T$ be the family of all $S$-$N$-independent subsets
in $M.$ Since $\{0\}\in T$ it follows that $T$ is  non-empty.
Let $D$ be any directed subfamily of $T$ and let $U$ be the union
of all members of $D.$ Then  $U\in T$  since every finite subset of $U$ is  $S$-$N$-independent.  By Zorn's lemma, $T$ has a maximal element,
say $W.$ Now we will prove that $N(W)=M.$ Assume that $N(W)\neq M,$
thus $N(W)$ is a proper $S$-pure submodule of $M.$ Hence there
exists a non-zero indecomposable submodule $B$ of $M$ such that
$N(W)\cap B=0$ and $N(W)+B=N(W)\oplus B=N\oplus(\underset{A\in W}{\sum}\oplus A)\oplus B$
is an $S$-pure submodule of $M,$ as seen above. Hence $W\cup\{B\}$ properly contains $W$ and is  $S$-$N$-independent in $M.$ This contradicts the maximality of $W$
in $T.$ Therefore, $N(W)=M.$ Since $N(W)=N\oplus(\underset{A\in W}{\sum}\oplus A)$
it follows that $N$ is a direct summand of $M$ and $M/N\simeq\underset{A\in W}{\sum}\oplus A$
is a direct sum of indecomposable submodules. If we take $N=0$ then we see that
$M$ is a direct sum of indecomposable submodules.\end{proof}

\begin{cor}\label{cor:5.6(Theorem 4.24, proposition 4.25, Corollary 4.25(a), p.21(a))}
 Let S be any set of finitely presented
left R-modules. Then the following statements are equivalent:

$\left(1\right)$ Every indecomposable left R-module is S-pure-projective.

$\left(2\right)$ For any left R-module M, every S-pure  submodule
of M is a direct summand of M.

$\left(3\right)$ Every left R-module is S-pure-projective.

$\left(4\right)$ Every left R-module is S-pure-injective.

$\left(5\right)$ Every left R-module is a direct sum of modules in
\textup{ind}$(S\cup\{_{R}R\})$.

\end{cor}

\begin{proof}
$\left(1\right)\Rightarrow\left(2\right).$ Let $M$ be any left $R$-module
and let $N$ be any $S$-pure submodule of $M$ such that $M/N$ is
 indecomposable. By hypothesis, $M/N$ is  $S$-pure-projective hence
 the $S$-pure exact sequence $0\rightarrow N\overset{i}{\rightarrow}M\overset{\pi}{\rightarrow}M/N\rightarrow0$
splits and hence $N$ is a direct summand of $M.$ By Theorem~\ref{thm:5.5(Theorem 4.22, p.16)}, every $S$-pure submodule of $M$ is a direct summand of $M$.

$\left(2\right)\Rightarrow\left(3\right).$ Let $M$ be any left $R$-module
and let $\Sigma:0\rightarrow L\overset{f}{\rightarrow}N\overset{g}{\rightarrow}M\rightarrow0$
be any $S$-pure exact sequence of left $R$-modules.  By hypothesis, im$(f)$ is a direct
summand of $N$ and hence
$\Sigma$ is split so $M$ is $S$-pure-projective.

$\left(3\right)\Rightarrow\left(5\right).$ Assume that every left
$R$-module is $S$-pure-projective, thus every left $R$-module is
pure-projective. By \cite[Proposition 4.4, p.73]{Azu92}, $R$ is
a left Artinian ring and hence $R$ is  Krull-Schmidt, by \cite[p.164]{Pre09}.
Let $M$ be any left $R$-module. By hypothesis and Proposition~\ref{prop:4.1(4.12)(a),p.7(a)}, $M$ isomorphic to a direct
sum of modules in ind$(S\cup\{_{R}R\}).$ Thus every left $R$-module
is a direct sum of modules in ind$(S\cup\{_{R}R\})$.

$\left(5\right)\Rightarrow\left(1\right).$ Assume that every left
$R$-module is a direct sum of modules in ind$(S\cup\{_{R}R\})$.
Let $M$ be an indecomposable left $R$-module, thus $M$ is a direct
sum of modules in ind$(S\cup\{_{R}R\})$. Since each module in
ind$(S\cup\{_{R}R\})$ is $S$-pure-projective and the class of
$S$-pure-projective left $R$-modules is closed under direct sums
(by \cite[ p.278]{Wis91}) it follows that $M$ is $S$-pure-projective.
Hence every indecomposable left $R$-module is $S$-pure-projective.

$\left(2\right)\Leftrightarrow\left(4\right).$ By using \cite[33.7, p.279]{Wis91}.
\end{proof}

\subparagraph*{Acknowledgements:}

The results of this paper will be  part of  my doctoral thesis at the University of Manchester. I wish to express my sincere thanks to my supervisor Professor Mike Prest, for his valuable help and suggestions and for his help in the preparation of
this paper. Also, I would like to  express my thank and gratitude to the Iraqi government
for sponsoring and giving me the opportunity to study for a  PhD at the University of Manchester.

\end{document}